\documentclass{article}

\usepackage[preprint]{neurips_2023}

\usepackage[utf8]{inputenc} 
\usepackage[T1]{fontenc}    
\usepackage{hyperref}       
\usepackage{url}            
\usepackage{booktabs}       
\usepackage{amsfonts}       
\usepackage{nicefrac}       
\usepackage{microtype}      
\usepackage{xcolor}         
\usepackage{natbib}

\usepackage{microtype}
\usepackage{graphicx}
\usepackage{subfigure}
\usepackage{booktabs} 

\usepackage{hyperref}
\hypersetup{colorlinks,linkcolor={blue},citecolor={blue},urlcolor={red}}  

\usepackage{amsmath}
\usepackage{amssymb}
\usepackage{mathtools}
\usepackage{amsthm}
\usepackage{tablefootnote}
\usepackage{enumitem}

\usepackage[capitalize,noabbrev]{cleveref}

\theoremstyle{plain}
\newtheorem{theorem}{Theorem}[section]
\newtheorem{proposition}[theorem]{Proposition}
\newtheorem{lemma}[theorem]{Lemma}
\newtheorem{corollary}[theorem]{Corollary}
\theoremstyle{definition}
\newtheorem{definition}[theorem]{Definition}
\newtheorem{assumption}[theorem]{Assumption}
\theoremstyle{remark}

\newtheorem{example}[theorem]{Example}

\DeclareMathOperator*{\argmin}{arg\,min}

\def\la{\langle}
\def\ra{\rangle}
\def\a{\alpha}

\def\dv#1{{\color{black}#1}}
\def\dist{{\rm dist}}
\def\bd{m}
\def\bs{s}

\title{Last Iterate Convergence of Popov Method for Non-monotone Stochastic Variational Inequalities}

\author{%
  Daniil Vankov$^{1}$ \quad
  Angelia Nedi\'{c}$^{1}$ \quad
  Lalitha Sankar$^{1}$ \\
$^{1}$Arizona State University \\
\{dvankov, angelia.nedich, lsankar\}@asu.edu
}

\begin{document}

\maketitle

\begin{abstract}
This paper focuses on non-monotone stochastic variational inequalities (SVIs) that may not have a unique solution. A commonly used efficient algorithm to solve VIs is the Popov method, which is known to have the optimal convergence rate for VIs with Lipschitz continuous and strongly monotone operators. We introduce a broader class of structured non-monotone operators, namely \emph{$p$-quasi sharp} operators ($p> 0$), which allows tractably analyzing convergence behavior of algorithms. We show that the stochastic Popov method converges \emph{almost surely} to a solution for all operators from this class under a \emph{linear growth}. In addition, we obtain the last iterate convergence rate (in expectation) for the method under a \emph{linear growth} condition for $2$-quasi sharp operators. Based on our analysis, we refine the results for smooth $2$-quasi sharp and $p$-quasi sharp operators (on a compact set), and obtain the optimal convergence rates. We further provide numerical experiments that demonstrate advantages of stochastic Popov method over stochastic projection method for solving SVIs.

\end{abstract}

\section{Introduction}
Recently, the framework of variational inequalities (VIs) has attracted much attention from researchers due to the wide range of its applications. A VI problem results when generalizing a variety of optimization problems, including those involving  constraints, min-max optimization, and more general non-zero sum games. The adversarial approach in machine learning (ML) is yet another motivation behind the recent interest in stochastic VIs which allows modeling stochasticity in training. 

In constrained optimization problems, the \emph{projection method} (also known as the gradient method) is widely used.  However, it suffers from non-convergent behavior when applied to monotone VIs. To overcome this, \cite{korpelevich1976extragradient} proposed the \emph{extragradient (EG) method} and proved asymptotic convergence of the algorithm with linear convergence guarantees for a linear operator. Shortly after, \cite{popov1980modification} proposed a method that achieves the same result for monotone VIs as in \cite{korpelevich1976extragradient}. More recently, Popov's method, often referred to as the \emph{past extragradient (PEG) method}, and its variant, called \emph{optimistic gradient (OG) method}, have been shown to have optimal convergence rates for strongly monotone smooth VIs (\cite{beznosikov2022smooth}).

In ML applications, including optimization of deep neural networks or generative adversarial networks (GANs), a large condition number associated with the operator is a key source of slower convergence rates~(\cite{nachmani2018near}). When the condition number of the operator, defined as the ratio of the Lipschitz constant to the strong monotonicity constant, is large, it has been observed that the projection method is slower than Popov and EG algorithms (\cite{beznosikov2022smooth}). Furthermore, while both Popov and EG methods have the same theoretical upper bound on the number of iterations for monotone operators (\cite{cai2022tight}, \cite{gorbunov2021extragradient}), the Popov method requires only one oracle call per iteration, while EG requires two. For these reasons, we focus on the Popov method.

Most results on the last iterate convergence of first-order methods for the stochastic VI involve strong monotonicity. Weak sharpness, a weaker condition than strong monotonicity, is widely used to show convergences in optimization and monotone VI problems (\cite{yousefian2014optimal,kannan2019optimal}). But in many real-life applications (e.g., GANs where both the discriminator and generator usually are nonconvex deep neural networks), the resulting VI is not monotone. To address these challenges, in \cite{hsieh2020explore}, a tractable weaker condition of \emph{variational stability} is used to capture a large class of VIs, including monotone VIs. Recently, another class of structured non-monotone operators called \emph{quasi-strong monotonicity} was introduced in \cite{loizou2021stochastic}. While quasi-strong monotonicity has attracted attention recently (e.g, \cite{gorbunov2021stochastic, choudhury2023single}), most such approaches assume unconstrained VIs with a unique solution. 


We present a broad class of structured, non-monotone, constrained VIs with non-unique solutions, called $p$-quasi sharp. Figure~\ref{fig:venn} visualizes the relationship between our newly introduced class of operators and the existing ones. For this setting, we now summarize our contributions. 

\begin{figure}[ht]
\centering
\includegraphics[scale=0.2]{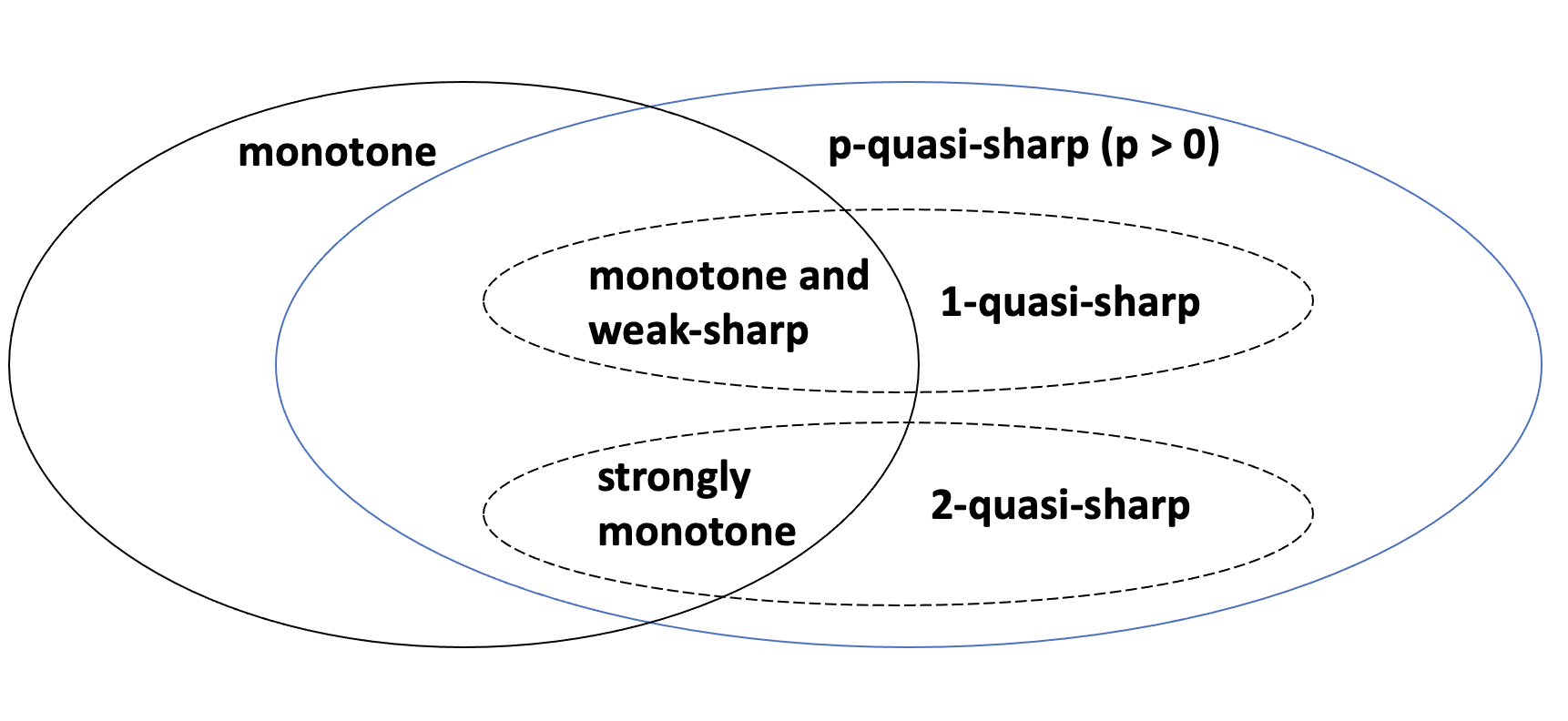}
\caption{Relations between different operator classes and the new class of $p$-quasi sharp operators.}
\label{fig:venn}
\end{figure}


\subsection{Our contributions}
We summarize our contributions on the last iterate convergence guarantees of {stochastic Popov} method for \emph{constrained non-monotone stochastic VIs with non-unique solutions} 
(see also Table~1).
\begin{itemize}
\vspace{-0.1in}
\setlength\itemsep{-0.1em}
    \item A key feature of our analysis is the use of a new type of non-monotone VIs with {special operators, termed} $p$-\emph{quasi sharp}, $p>0$. {The class of monotone and weak-sharp operators is contained in the class of $p$-quasi sharp operators (see Figure~\ref{fig:venn}). Moreover, when the VI solution is not unique,  $p$-quasi sharpness is a weaker condition than that of quasi-strong monotonicity.} \dv{Also, we present an example of an operator that satisfies $p$-quasi sharpness, is not monotone, and does not satisfy previously considered assumptions, such as quasi-strong monotonicity and weak sharpness.}
    \item Our main contribution is proving almost sure (\emph{a.s.}) convergence when the operator is assumed to have a \emph{linear growth} and $p$-\emph{quasi sharp}, for all $p$. The class of linear growth operators includes Lipschitz continuous and bounded operators. \emph{We prove {a.s.} convergence of the iterates to a solution}, which is a  new result for this setting in contrast with the existing results showing only the convergence of the iterate distances to the solution set. To the best of our knowledge, this is the most general result on \emph{a.s.} convergence of the stochastic Popov method.
    


    \item Our second main contribution is in deriving the \emph{first known} 
    last iterate sublinear 
    convergence rates for linear growth operators under the $p$-quasi sharpness assumption with $p=2$. 
    Also, leveraging results from \cite{stich2019unified} we obtain $\mathcal{O}\left(C^2 R_0 \exp \left( - \mu^2 K / C^2 \right) / \mu^2 + \sigma^2 / \mu^2 K \right)$ when the number $K$ of iterations is given.
    We then refine the analysis to Lipschitz continuous operators and obtain $\mathcal{O}\left(L   R_0 \exp \left( -\mu K / L   \right)/ \mu+ \sigma^2/ \mu^2 K \right)$ convergence rate. 
    This rate also holds for the quasi-strongly monotone setting wherein we recover the result from \cite{choudhury2023single} for unconstrained {finite-sum} VIs.

    \item Finally, we focus on convergence rates under $p$-quasi sharpness with $p\leq 2$. We derive rate bounds for the VIs with a compact constraint set and a continuous operator. In this setting, we 
    obtain asymptotic convergence and the last iterate convergence rate in the order of $\mathcal{O}\left(R_0 \exp \left( - K \right) + (\sigma^2 + D^2)M_U^{2(2 - p)} / \mu^2 K \right)$.
    

\end{itemize}

\subsection{Related Work}

\begin{table*}
\scriptsize
    \begin{center}
    \label{tab:table1}
    \begin{tabular}{|l|l|}
        \hline
       \textbf{Assumptions on operator $F(\cdot)$} &  \textbf{Rates} \\
        \hline
        $\|F(u)\|\le C\|u\| +D$,\ \  $p\in (0,\infty)$  & Asymptotic Convergence (Our Thm \ref{Theorem-Linear-sharp-ASC}) \\
        \hline
         $\|F(u)\|\le C\|u\| +D$,\ \  $p=2$ & $ \frac{C^2}{\mu^2} R_0 \exp [-\frac{\mu^2}{C^2} K] +  \frac{\sigma^2}{\mu^2 K}$(Our Thm \ref{Theorem-Linear-quasi-rate-exp})  \\
      \hline
     Lipschitz continuous, \ \ $p=2$ &  $ \frac{ \sigma^2 }{\mu^2K}$  \cite{hsieh2019convergence} \\
     & $ \frac{L}{\mu} R_0 \exp [-\frac{\mu}{L} K] + \frac{\sigma^2}{\mu^2 K}$ \cite{choudhury2023single} \\
       & $ \frac{L}{\mu} R_0 \exp [-\frac{\mu}{L} K] + \frac{\sigma^2}{\mu^2 K}$ (Our Thm \ref{Theorem-Lipschitz-quasi-rate-exp})\\
     \hline
     $\|F(u)\| \leq D$, \ \ $\textcolor{black}{p \leq 2}$ &  $R_0 \exp [- K] + \frac{(\sigma^2 + D^2) M_2^{\textcolor{black}{2(2 - p)}}}{\mu^2 K}$  (Our Thm \ref{Theorem-Linear-sharp-rate}) \\ 
     \hline
    \end{tabular}
    \end{center}
    \caption{ \scriptsize  Summary of the best known and our results on convergence rates of stochastic Popov method for stochastic VIs  under $p$-quasi sharpness assumption (see Assumption~\ref{asum-sharp}). 
    {Convergence rates are obtained for the case when the number $K$ of iterations is given and fixed. When the number $K$ of iterations is not given,  then the uniform upper bounds on the error after $k$ iterations are $\mathcal{O}(1 / k)$ in all cases}. Paper~\cite{hsieh2019convergence} provides a convergence rate only for strongly-monotone unconstrained SVIs, while~\cite{choudhury2023single} has a convergence rate for finite-sum unconstrained VIs.
    }
\end{table*}

The first analysis of the optimal first-order method to solve monotone SVI on a compact set was presented in \cite{juditsky2011solving}, where authors proposed a variant of EG with general Bregman projections. They proved average iterates convergence to the solution set when the operator is Lipschitz continuous or bounded. Later \cite{hsieh2019convergence} studied single oracle call extragradient methods, including Popov and the OG methods. Using the Lipschitzness and strong monotonicity of the operator, they showed $\mathcal{O}(1/ k)$ last iterate convergence in stochastic unconstrained VI. In the following work, \cite{hsieh2020explore} obtained \emph{a.s.}\  convergence of EG method for stochastic unconstrained VI under variational stability condition and convergence rate of $\mathcal{O}(1/ k^{1/3})$ with an additional error bound assumption. 

In \cite{yousefian2014optimal, kannan2019optimal}, the authors obtained sublinear convergence rates of EG method for monotone SVIs under weak-sharpness assumption on a compact set. The same rate was delivered using the strong pseudomonotonicity property of the operator in \cite{kannan2019optimal}. Later, in \cite{song2020optimistic}, it was shown that strong pseudomonotonicty is a stronger assumption than quasi-strong coherent condition. 

Recently, quasi-strong monotonicity was introduced in \cite{loizou2021stochastic} to establish last iterate convergence rates of  projection and consensus methods for unconstrained SVIs with unique solutions. For the same setting, both \cite{gorbunov2021stochastic} and \cite{choudhury2023single} studied EG and Popov methods, respectively, under a quasi-strong monotonicity condition and derived $\mathcal{O}(L  R_0 \exp [-\mu K / L] / \mu + \sigma^2 / \mu^2 K)$ convergence rates. 
\dv{To the best of our knowledge, the weakest class of structured non-monotone operators, named weak Minty VI (also called co-monotonicity), was proposed in \cite{diakonikolas2021efficient}. It has been shown in \cite{yoon2021accelerated} that the optimal convergence rate for such class is sublinear, while under quasi-strong monotonicity it is linear.}

\dv{In our work, we straddle between the abovementioned two  assumptions on the operator via  $p$-quasi sharp operator class which contains quasi-strong monotonic operators and is contained within weak Minty VI operator class}. We consider constrained SVIs with non-monotone operators under linear growth; we do not assume uniqueness of the solution. 
The rest of the paper is organized as follows. In Section \ref{Preliminaries}, we introduce a general VI problem and provide our assumptions.
In Section \ref{Convergence}, we present our main results on the last iterate convergence and provide convergence rates. We illustrate our results in Section \ref{Numerical} with simulations and present a discussion in Section \ref{Discussion}.

\section{Variational Inequality Problem}
\label{Preliminaries}
The solutions to convex optimization problems and convex-concave saddle point problems, as well as 
more general games, can be characterized as solutions to related variational inequality problems. A variational inequality problem is specified by a (nonempty) set 
$U\subseteq\mathbb{R}^{\bd}$ and an operator $F(\cdot): U\to \mathbb{R}^{\bd}$, and denoted by VI$(U,F)$. For $U=\mathbb{R}^{\bd}$, we obtain an unconstrained VI.
The variational inequality problem VI$(U,F)$ consists of determining a point $u^* \in U$ such that
\begin{equation}
\begin{aligned}
\label{VI}
\langle F(u^*), u - u^* \rangle \geq  0 \qquad \hbox{for all } u \in U. 
\end{aligned}
\end{equation}
The solution set for the VI$(U,F)$ is denoted by $U^*$,
i.e.,
\[U^*=\{u^*\in U\mid \langle F(u^*), u - u^* \rangle \geq  0 \ \ \hbox{for all } u \in U\}.\]
An operator $F(\cdot):U\to\mathbb{R}^\bd$ is said to be Lipschitz continuous over a set $U$ if there exists scalar $L > 0$ such that
$$\|F(u) - F(u') \| \leq L\|u - u'\|\qquad\hbox{for all $u, u' \in U$}.$$
An operator $F(\cdot)$ is said to be \dv{ \it $p$-monotone} \dv{for $p \geq 1$} (\cite{facchinei2003finite})   over the set $U$, 
if there exists some $\mu>0$ such that
\begin{equation}
\label{eq-p-monotone}
\langle F(u)-F(v),u-v\rangle \ge \mu \|u-v\|^{\textcolor{black}{p}}\quad\hbox{for all }u,v\in U.
\end{equation}
When the preceding relation holds with $\mu=0$, the operator $F(\cdot)$ is said to be {\it monotone} on the set $U$. When it holds with $p=2$, the operator $F(\cdot)$ is said to be {\it strongly monotone} on the set $U$.

{For a convex closed set $U$, and a strongly monotone and Lipschitz continuous operator $F(\cdot)$, 
the VI$(U,F)$ can be solved by a projection method with update rule
\begin{equation}
\begin{aligned}
\label{eq-proj-det}
u_{k+1} =  P_{U}(u_k - \alpha F(u_k)),
\end{aligned}
\end{equation}
where $\alpha>0$ is a stepsize, $u_0\in U$ is an arbitrary initial point, 
and $P_U(\cdot)$ is the projection operator on the set $U$,
i.e., $P_U(u) = \argmin_{v \in U} \| v - u \|$. When the operator is just monotone,
the projection method may not work. In this case,
Popov method (\cite{popov1980modification}) can be used, which is given as follows:
for all $k\ge0$,
\begin{equation}
\begin{aligned}
\label{eq-popov-det}
u_{k+1} &=  P_{U}(u_k - \alpha F(h_k)), \cr
h_{k+1} &= P_{U}(u_{k+1} - \alpha  F(h_{k})),   \\
\end{aligned}
\end{equation}
where $\a>0$ is a stepsize and $u_0,h_0\in U$ are arbitrary initial points.
}

{
We focus on a stochastic variational inequality problem, denoted by SVI$(U,F)$, corresponding to the case when
the operator $F(\cdot)$ is specified as the expected value of a stochastic operator 
$\Phi(\cdot,\cdot): U \times \mathbb{R}^d \rightarrow \mathbb{R}^\bd$, i.e.,
\[F(u) = \mathbb{E}[\Phi (u,\xi)]\qquad\hbox{for all }u\in U,\]
where $\xi$ is a $d$-dimensional random vector. For such a problem,
we consider a stochastic variant of the Popov method~\eqref{eq-popov-det}, where a stochastic approximation $\Phi(h_k, \xi_k)$ is used instead of the direction $F(h_k)$. Specifically, at each iteration $k$, the algorithm generates two iterates $u_{k+1}$ and $h_{k+1}$. Having the current iterate $h_k$, a random sample $\xi_k$ is drawn according to the distribution of the random variable $\xi$,
and the updates are defined by:
\begin{equation}
\begin{aligned}
\label{eq-popov-stoch}
u_{k+1} &=  P_{U}(u_k - \alpha_k \Phi(h_k, \xi_k)), \cr h_{k+1} &= P_{U}(u_{k+1} - \alpha_{k+1}  \Phi(h_{k}, \xi_k)), \\
\end{aligned}
\end{equation}
where $\alpha_k>0$ is a stepsize, and $u_0, h_0 \in U$ are arbitrary deterministic points\footnote{The results easily extend to the case when the initial points are random as long as $\mathbb{E}[\|u_0\|^2]$ and $\mathbb{E}[\|h_0\|^2]$ are finite.}. 
} 

{
Regarding the stochastic approximation error $\Phi(h_k, \xi_k)-F(h_k)$, we assume that it is unbiased and with a finite variance, formalized as follows.
\begin{assumption}\label{asum-samples}
The random sample sequence $\{\xi_k\}$ is such that for some $\sigma>0$ and for all $k\ge0$,
\begin{align*} &\mathbb{E}[\Phi(h_k, \xi_k)-F(h_k)\mid h_k]= 0, \cr 
&\mathbb{E}[\|\Phi(h_k, \xi_k)-F(h_k)\|^2\mid h_k]\le \sigma^2.
\end{align*}
\end{assumption}
}

Regarding the VI$(U,F)$, we will assume that the set $U$ is closed and convex, 
which is required in order to have the projection on $U$ uniquely defined,
so the method~\eqref{eq-popov-stoch} is well defined.
We will also assume
that the solution set $U^*$ is nonempty and closed.
For example, when $F(\cdot)$ is continuous, the set $U^*$ is closed. However,
we do not assume the continuity of $F(\cdot)$.

We provide convergence analysis of the stochastic Popov method~\eqref{eq-popov-stoch} under different assumptions on the behavior of the operator $F(\cdot)$, including a linear growth that allows upper bounding the growth of $\|F(u)\|$ as linear in $\|u\|$, defined as follows.

%

\begin{definition}
\label{def-lingrowth}
Given a set $Y\subseteq\mathbb{R}^\bd$ and an operator $G(\cdot):Y\to\mathbb{R}^\bd$, we say that $G(\cdot)$ has a linear growth on the set $Y$\ if there exist scalars $C\ge0$ and $D \geq 0$ such that 
$$\|G(y) \| \leq C \|y\| + D\qquad\hbox{for all }y\in Y.$$
\end{definition}
An operator $G(\cdot)$ is {\it bounded on the set $Y$} if the preceding linear growth condition is satisfied with $C=0$.

A continuous operator $G(\cdot):Y\to\mathbb{R}^\bd$ has a linear growth over a compact set $Y$. Moreover,
when operator $G(\cdot)$ is Lipschitz continuous over the set $Y$, then it has a linear growth on $Y$. To see this,  let 
$y'\in Y$ be an arbitrary fixed point. Then, we have
for all $y\in Y$,
$$\|G(y) \| - \|G(y')\| \leq \|G(y) - G(y') \|  \leq L \|y\| + L\|y'\|,$$
implying that 
$\|G(y)\| \leq L \|y \| + L\|y'\| + \|G(y')\|$ for all $y\in Y$. Hence, $G(\cdot)$ has a linear growth on $Y$ with constants $C=L$ and $D=L\|y'\| + \|G(y')\|$.

{
The linear growth condition is formally imposed, as follows.
\begin{assumption}\label{asum-Linear} 
The operator $F(\cdot):U\to\mathbb{R}^\bd$ has a linear growth on the set $U$.
\end{assumption}
}

Additionally, we consider the $p$-quasi sharpness property which captures the behavior of the operator with respect to the solution set $U^*$. 
To provide this condition, for a point $u\in\mathbb{R}^\bd$ and 
a nonempty set $Y\subseteq\mathbb{R}^\bd$, we use $\dist(u,Y)$ to denote the distance from $u$ to the set $Y$, i.e.,
$\dist(u,Y)=\inf_{y\in Y}\|u-y\|$.
{We note that, for any nonempty set 
$Y\subseteq\mathbb{R}^\bd$, {\it the distance function $\dist(\cdot,Y)$ is continuous\footnote{It is Lipschitz continuous with Lipschitz constant 1.}.
}

The $p$-quasi sharpness property is defined as follows.
\begin{definition}\label{def-sharp} 
Given two sets $U,Y\subseteq\mathbb{R}^\bd$,
an operator $G(\cdot):U\to\mathbb{R}^\bd$ has a $p$-quasi sharpness property over the set $U$ relative to $Y$, with constants $p>0$ and $\mu >0$,
if the following relation holds
for all $u \in U$ and all $y \in Y$,
\[
\langle G(u), u - y \rangle \geq \mu \,\dist^p(u, Y).
\]
\end{definition}

\begin{assumption}\label{asum-sharp} The operator $F(\cdot):U\to\mathbb{R}^\bd$ has a $p$-quasi sharpness property over $U$ relative to the solution set $U^*$, i.e.,
for some $p>0$, $\mu >0$, and for all $u \in U$ and all $u^* \in U^*$,
\begin{equation}\label{eq-psharp}
\langle F(u), u - u^* \rangle \geq \mu \,\dist^p(u, U^*).
\end{equation}
\end{assumption} 
Note that the relation~\eqref{eq-psharp}
is a generalization of the $p$-monotone operator~\eqref{eq-p-monotone}. Also,  note that an operator with the $p$-quasi sharpness property need not be monotone.
When $p=1$, the $1$-quasi sharpness property 
is different from the 
weak-sharpness property which requires that the relation~\eqref{eq-psharp} holds with $F(u^*)$ instead of $F(u)$.
The weak-sharpness property was used in~\cite{yousefian2014optimal}, \cite{kannan2019optimal} to show almost sure convergence and optimal convergence rate for the extragradient method in a Nash game setting.
We note that a monotone operator with the weak-sharpness property has the $p$-quasi sharpness property with $p=1$. 

\dv{When $p \geq 2$, the $p$-quasi sharpness implies Saddle-Point Metric Subregularity (SP-MS) condition considered in \cite{wei2020linear}. However, we consider a more general VI than the VI arising from a 
convex-concave min-max game considered in~\cite{wei2020linear}.
}
When $p=2$, the $2$-quasi sharpness property
includes the quasi-strongly monotone property,
which has been used in~\cite{loizou2021stochastic}, \cite{gorbunov2021stochastic} to analyze the convergence of stochastic gradient and extra-gradient methods.
It has been shown in \cite{loizou2021stochastic} that the quasi-monotonicity ($p=2$ and $\mu=0$) property is weaker than monotonicity of the operator. We note that an operator can posses $p$-quasi sharpness property  but need not necessarily be monotone. Leveraging on this property, we can show convergence results for non-monotone SVIs (and VIs). 

In~\cite{song2020optimistic}, a strong coherent property is used,
which corresponds to $2$-quasi sharpness property ($p=2$) when  $\dist^2(u,U^*)$ in~\eqref{eq-psharp} is replaced by $\|u - u^*\|^2$.
When the VI$(U,F)$ has a unique solution, the strong coherent and our 2-quasi sharpness properties are equivalent.  However, when the solution set $U^*$ is not a singleton, the 2-quasi sharpness property is weaker than the strong coherent property. It has been shown in~\cite{song2020optimistic} that strong coherent property is weaker than the strong pseudo-monotonicity which was used in~\cite{kannan2019optimal}. 

Next, we present an example of operator that satisfies $p$-quasi sharpness but it does not satisfy the conditions previously considered in the existing literature.

\begin{example} \label{example-p}
Let $p>0$ and operator $F: \mathbb{R}^2 \rightarrow \mathbb{R}^2$ be defined as
\begin{align}
    F(u) = c \begin{bmatrix} \text{\rm sign}(u_1) |u_1|^{p - 1} + u_2 \\ \text{\rm sign}(u_2) |u_2|^{p - 1} - u_1 \end{bmatrix}, c= \left\{ \begin{array}{@{}cc} 2,  \|u\| \leq 1, \\ 1,  \|u\| > 1. \end{array} \right.
\end{align}
Then $F$ is $p$-quasi monotone with $\mu=2^{1 - p}$ and has a linear growth for $p\geq 2$. However, $F$ is not monotone and it does not satisfy the assumptions of positive co-monotonicity or quasi-strong monotonicity for any $p \in [1, 2) \cup (2, \infty)$. Furthermore, operator $F$ is not Lipschitz continuous. 
\end{example}


\dv{We rigorously prove the above observations in Appendix \ref{appendix-A}.
We also present visualization of the vector fields of $F(u)$ and numerical experiments for different values of $p$ in Appendix  \ref{Numeric-Details}.
To connect the seemingly different assumptions of linear growth and $p$-quasi sharpness, in the following proposition, we identify when an operator cannot satisfy both assumptions.} 
\dv{\begin{proposition}
\label{prop-1}
Let the operator $F(\cdot):U\to\mathbb{R}^m$ be $p$-quasi sharp over an unbounded set $U\subseteq\mathbb{R}^m$ with $p>2$.
Assume that  VI$(U,F)$ has a compact solution set $U^*$. 
Then, operator $F(\cdot)$ does not have linear growth on $U$.
\end{proposition}}

\section{Last Iterate Convergence Analysis}
\label{Convergence}

In this section, we present our convergence analysis of stochastic Popov method for solving for SVI$(U,F)$. We first provide a lemma presenting the main inequality for the iterates of stochastic Popov method~\eqref{eq-popov-stoch} without any assumptions on the operator. This lemma is the basis for all the subsequent results.
\begin{lemma} \label{Lemma1} 
Let $U$ be a closed convex set. Then, for the iterate sequences
$\{u_k\}$ and $\{h_k\}$ generated by the stochastic Popov method~\eqref{eq-popov-stoch} we surely have
for all $y \in U$ and $k \ge 1$,
\begin{equation*}
\begin{aligned}
 \|u_{k+1}  &- y\|^2  \leq \|u_k  - y\|^2 - \|u_{k+1} -h_k\|^2 - \|u_k - h_k\|^2 \cr
 &- 2 \alpha_{k} \langle e_k+ F(h_k), h_k - y\rangle + 6\alpha_{k}^2 \|e_{k-1}\|^2 \\
&+ 6 \alpha_{k}^2 \,\|F(h_k) -  F(h_{k-1}) \|^2
+ 6\alpha_{k}^2 \|e_k\|^2 ,
\end{aligned}
\end{equation*}
where  $e_{k} = \Phi(h_k, \xi_k) - F(h_k)$ for all $k \ge 0$.
\end{lemma}
Lemma \ref{Lemma1} can be further refined for different types of operators, such as Lipschitz continuous and operators with linear growth.

\subsection{Almost Sure and in-Expectation Convergence}
In this section, we establish almost sure (\emph{a.s.}) convergence of the stochastic Popov method for SVI$(U,F)$ assuming that
the stepsize is diminishing.
\begin{assumption}\label{asum-steps} The positive sequence $\{\alpha_k\}$ is such that
\begin{equation}
    \begin{aligned}
     \alpha_k>0\ \hbox{for all }k,\quad
     \sum_{k=0}^{\infty} \alpha_k = \infty,\quad \quad \sum_{k=0}^{\infty} \alpha_k^2 < \infty.
    \end{aligned}
\end{equation}
\end{assumption}

The following theorem shows \emph{a.s.}\ convergence of the method
when the operator
$F(\cdot)$
has a linear growth and the $p$-quasi sharpness property. 
A key challenge in our setting is the construction of the appropriate almost super-martingale required to prove convergence.


\begin{theorem} \label{Theorem-Linear-sharp-ASC}
Let Assumptions~\ref{asum-samples}, \ref{asum-Linear}, \ref{asum-sharp}, and~\ref{asum-steps} hold. 
Then, the following statements hold the iterate sequences
$\{u_k\}$ and $\{h_k\}$ generated by the stochastic Popov method~\eqref{eq-popov-stoch}: 
\begin{itemize}
\item[(a)]
The sequence $\{u_k\}$ and $\{h_k\}$ converge almost surely to some $\bar u \in U$, where $\bar{u}$ is a solution \emph{a.s.}, i.e.  ${\rm Prob}(\bar u \in U^*) = 1$. 
\item[(b)] The sequences $\{\mathbb{E}[\|u_k\|^2]\}$ and 
 $\{\mathbb{E}[\|h_k\|^2]\}$ are bounded. 
\item[(c)] If the solution set $U^*$ is bounded, then 
the sequences $\{u_k\}$ and $\{h_k\}$ also converge in expectation, i.e.,
\[\lim_{k\to\infty}\mathbb{E}[\|u_k-\bar u\|^2]=0,\qquad \lim_{k\to\infty}\mathbb{E}[\|h_k-\bar u\|^2]=0,\]
\end{itemize} 
\end{theorem}

The proof of Theorem~\ref{Theorem-Linear-sharp-ASC} 
is given in Appendix ~\ref{sec-proof-Theorem-Linear-sharp-ASC}.
It is based on Lemma~\ref{Lemma1} and the $p$-quasi sharpness property.
Then, we rely on an almost super-martingale result (\cite{robbins1971convergence}, see also~\cite{polyak1987introduction}, Lemma~11) 
to prove part~(a), and a deterministic variant of that result to prove part~(b).
To establish part~(c), we use part~(b) and the Lebesgue dominated convergence theorem (e.g., Theorem 16.4 in~\cite{billingsley}).

In the existing works on the stochastic first-order methods, such as EG and projection methods for quasi-strongly monotone operators (\cite{gorbunov2021stochastic}, \cite{loizou2021stochastic}), there are no results on  \emph{a.s.}\ convergence of the iterates to a solution, except for the case when the solution set $U^*$ is a singleton. Our Theorem~\ref{Theorem-Linear-sharp-ASC} shows that such {\it a.s.}\ convergence results are possible even when the solution set $U^*$ is not necessarily a singleton.

As a direct consequence of Theorem~\ref{Theorem-Linear-sharp-ASC}, 
and the fact that Lipschitz continuous operators satisfy the linear growth condition, we have the following result.
\begin{corollary}
Let $U\subseteq\mathbb{R}^\bd$ be a closed convex set and $U^*\ne\emptyset$. Assume that the operator 
$F(\cdot):U\to\mathbb{R}^\bd$ is Lipschitz continuous on the set $U$.
Also, let Assumptions~\ref{asum-samples},
\ref{asum-sharp}, and~\ref{asum-steps} hold.
Then, the iterates $u_k$ and $h_k$ from stochastic Popov method converge {\it a.s.}\ to a (random) solution $v^* \in U^*$.
\end{corollary}

\subsection{Convergence Rates}\label{sec-rates}
Here, we present convergence rate results for the stochastic Popov method when the operator $F(\cdot)$ has \dv{$p$-quasi sharpness property with $p\leq2$.}

\textbf{2-quasi Sharp $F(\cdot)$}
We explore convergence rate for different stepsize choices. For 
a diminishing stepsize of the form $\alpha_k = \frac{2}{\mu(1 + k)}$, we have and show a sublinear convergence rate, as given in the following theorem.

\begin{theorem} \label{Theorem-Linear-quasi-rate-sub}
Let Assumptions~\ref{asum-samples}, \ref{asum-Linear}, and~\ref{asum-sharp} with $p=2$ hold. Also, let the stepsizes be given by 
$\alpha_k = \frac{2}{\mu(3 + k)}$ for all $k\ge0$. 
Then, the following inequality holds for the iterate sequence
$\{u_k\}$ generated by the stochastic Popov method~\eqref{eq-popov-stoch} for  all \dv{$K \ge 2$},
\begin{align*}
\mathbb{E}[\dist^2(u_{K+1}, U^*)]
\leq &\frac{18}{(K-1)^2}\,\mathbb{E}[\dist^2(u_1, U^*)] \cr
&+ \frac{24 (4 C^2 M + \sigma^2 + 2 D^2)}{\mu^2 (K-1)},
\end{align*}
where $C$ and $D$ are the constants from the linear growth condition (Assumption~\ref{asum-Linear}), while the scalar $M>0$ is such that $\mathbb{E} [\|h_k\|^2] \leq M$ for all $k \geq 1$.
\end{theorem}
The proof of Theorem~\ref{Theorem-Linear-quasi-rate-sub} relies on some basic properties of the method (given in Lemma~A.2
in the appendix), Lemma~7 from~\cite{stich2019unified}, and Theorem~\ref{Theorem-Linear-sharp-ASC}(b).

The convergence rate result of Theorem~\ref{Theorem-Linear-quasi-rate-sub} has a sublinear coefficent $1/K^2$ with $\mathbb{E}[\dist^2(u_1, U^*)]$. An exponential decay factor for the quantity $\mathbb{E}[\dist^2(u_1, U^*)]$ can be obtained with a different stepsize choice. In particular, to establish such a convergence rate result,
we use Lemma~3 of~\cite{stich2019unified} and the stepize choice given in the proof of that lemma, namely,
%
for any given $K\ge0$, the stepsize $\a_k$, $0 \leq k \leq K$ is given by
\begin{align}
\label{eq-lemma3stitchstep}
\a_k &= \frac{1}{d}\qquad \text{if } K \leq \frac{d}{a},\cr
\a_k &= \frac{1}{d}\qquad \text{if } K> \frac{d}{a} \ \mbox{ and } \ k < k_0,\\
\a_k &= \frac{2}{a \left( \frac{2d}{a} + k - k_0\right)}\qquad \text{if } K > \frac{d}{a} \ \mbox{ and } \ k \geq k_0,\nonumber
\end{align}
where $k_0 = \lceil{\frac{K}{2}}\rceil$ and $d \ge a > 0.$

Our convergence rate estimate with such a stepsize selection, summarized in the following theorem, is obtained assuming that the solution set $U^*$ is compact.

\begin{theorem}
\label{Theorem-Linear-quasi-rate-exp}
Let $U^*$ be a nonempty compact set.
Also, let Assumptions~\ref{asum-samples}, \ref{asum-Linear}, and~\ref{asum-sharp} with $p=2$ hold. For a given \dv{$K\ge 2$}, let the stepsize $\alpha_k$ be given as in~\eqref{eq-lemma3stitchstep} with  $a = \frac{\mu}{2}$ and $d$ satisfying
\[d^{-1}\le \min\left\{ \frac{\mu}{288 C^2},\frac{4}{9\mu}\right\},\]
where the constant $C>0$ is from the linear growth condition (Assumption~\ref{asum-Linear}).
Then, the following relation holds for the iterate sequence $\{u_k\}$ generated by the stochastic Popov method~\eqref{eq-popov-stoch} for all \dv{$K\ge 2$},
\begin{align*}
\mathbb{E}[\dist^2(u_{K+1}, U^*)]
\leq & 
\frac{64d}{\mu} r_1 e^{-\frac{\mu(K-1)}{4 d}} + \frac{144c}{\mu^2(K-1)},
\end{align*}
where 
\[r_1=\mathbb{E}[\dist^2(u_1, U^*) + \|h_{0} - u_1\|^2],\] 
\[c=12 \sigma^2 + 2 D^2 + 12 M_1^2,\]
and $M_1$ is an upper bound for the norms of solutions $u^*\in U^*$, i.e., $\|u^*\| \leq M_1$ for all $u^* \in U^*$. 
\end{theorem}

As noted earlier, a Lipschitz continuous operator $F(\cdot)$ on the set $U$ with a Lipschitz constant $L$ satisfies a linear growth condition with $C=L$ and $D=L\|u'\| + \|F(u')\|$ where $u'\in U$ is an arbitrary but fixed point. Thus, Theorem~\ref{Theorem-Linear-quasi-rate-exp} applies with $C=L$ to such an operator. 
By directly applying Theorem~\ref{Theorem-Linear-quasi-rate-exp} to
a Lipschitz continuous operator, we would obtain 
a convergence rate estimate of the form 
$\frac{L^2}{\mu^2} \tilde C_1 r_1 e^{-\frac{\mu^2}{L^2} K} +  \frac{L^2\tilde C_2}{\mu^2 K}$
for some positive constants $\tilde C_1$ and $\tilde C_2$ independent of $L$ and $\mu$.
However, a better convergence rate result can be obtained of the form $\frac{L}{\mu} \hat C_1 e^{-\frac{\mu}{L} K} +  \frac{\sigma^2\hat C_2}{\mu^2 K}$, where the positive constants 
$\hat C$ and $\hat C$ are independent of $L$ and $\mu$. We establish such an estimate by directly exploiting the Lipschitz  continuity of the operator, which also allows us to relax  the boundedness assumption for the solution set $U^*$ imposed in Theorem~\ref{Theorem-Linear-quasi-rate-exp},
as seen in the following theorem.
\begin{theorem} 
\label{Theorem-Lipschitz-quasi-rate-exp}
Let Assumption~\ref{asum-samples} hold, and
assume that the 
operator $F(\cdot)$ is Lipschitz continuous over $U$ with a constant $L>0$ and satisfies Assumption~\ref{asum-sharp} with $p=2$.
For any given \dv{ $K\ge2$}, let the stepsizes $\a_k$ be defined by~\eqref{eq-lemma3stitchstep}
with  $a = \mu$ and 
\[d\ge \max\left\{2\sqrt{3} L,\, \mu\right\}.\]
Then, the iterate sequence $\{u_k\}$ generated by the stochastic Popov method~\eqref{eq-popov-stoch}
satisfies the following inequality for all \dv{ $K\ge2$ },
\[\mathbb{E}\left[\dist^2(u_{K+1}, U^*) \right]\leq
\frac{32d}{\mu} r_1 e^{-\frac{\mu(K-1)}{2 d} } + \frac{432 \sigma^2 }{\mu^2(K-1)},
\]
where $r_1=\mathbb{E}\left[\dist^2(u_1, U^*) + \|h_{0} - u_1\|^2\right]$.
\end{theorem}

\textbf{$p$-Quasi Sharp Bounded $F(\cdot)$ with ($p \le 2$)}
Here, we establish a convergence rate result for the SVI$(U,,F$ with the operator $F(\cdot)$ that has $p$-quasi sharpness property with $p \leq 2$. For this result, we assume that the set $U$ is compact. A sublinear convergence result was shown in~\cite{yousefian2014optimal} for the EG under the assumptions of strict monotonicity, Lipschitz continuity,  
and the weak-sharpness property of the operator $F(\cdot)$, i.e.,
\[\la F(u^*),u-u^*\ra \ge \mu\, \dist(u,U^*)\quad\hbox{for all }u\in U, u^*\in U^*.\]
We, however, use a weaker condition, namely, $p$-quasi sharpness for the operator, and obtain a sublinear for the Popov method, as given in the following theorem.

\begin{theorem} \label{Theorem-Linear-sharp-rate}
Let operator $F(\cdot)$ be continuous. 
Let the constants $M_U>0$ and $D>0$ be such that $\|u-u'\|\leq M_U$ for all $u,u'\in U$
and $\|F(u)\|\le D$ for all $u\in U$.
Also, let Assumptions~\ref{asum-samples} and~\ref{asum-sharp} with \dv{$p\leq 2$} hold.
Then, for the iterate sequence $\{u_k\}$ generated by the stochastic Popov method~\eqref{eq-popov-stoch} the following statements are valid:
\begin{itemize}
    \item[(a)]
    Let stepsizes be given by \dv{$\alpha_k = \frac{2 M_U^{2-p}}{\mu(3 + k)}$} for all $k\ge0$. Then, we have for all \dv{ $K \geq 2$},
    \dv{\begin{align*}
    \mathbb{E}[\dist^2(u_{K+1}, U^*) ] &\leq  
\frac{32 }{(K-1)^2} \mathbb{E}[\dist^2(u_1, U^*)] \cr
&+ \frac{24(\sigma^2 + 2 D^2)M^{2(2-p)}_U}{\mu^2(K-1)}.
\end{align*}}
    \item[(b)]
    For any \dv{ $K\ge 2$ }, let the stepsizes $\a_k$ be given by~\eqref{eq-lemma3stitchstep} with  \dv{$a=\frac{\mu}{ M_U^{2-p}},$} and \dv{$d =\frac{2 \mu}{M_U^{2-p}}$}. Then, we have for all \dv{ $K\ge 2$}, 
\dv{\begin{align*}
\mathbb{E}[\dist^2(u_{K+1}, U^*) ]
 &\leq   
64\mathbb{E}[\dist^2(u_1, U^*)] e^{-\frac{(K-1)}{4}} \cr
&+ \frac{432 (\sigma^2 + 2 D^2) M_U^{2(2-p)}}{\mu^2(K-1)}.
\end{align*}}
\end{itemize}
\end{theorem}
It is worth comparing the rate results in Theorems \ref{Theorem-Linear-quasi-rate-exp} and \ref{Theorem-Lipschitz-quasi-rate-exp} with that in Theorem \ref{Theorem-Linear-sharp-rate} for $p=2$. We observe that the constant in the exponential term in Theorem~\ref{Theorem-Linear-sharp-rate}(b) depends only on the total number of iteratons $K$. In contrast, the corresponding terms in Theorems~\ref{Theorem-Linear-quasi-rate-exp} and \ref{Theorem-Lipschitz-quasi-rate-exp} exhibit dependencies on such parameters via $\mu/d < 4/9$ and $\mu/L$ respectively.  


\section{Numeric Results}
\begin{figure*}[hbt!]
\centering
\subfigure[$\kappa_F=10.1$]{
\includegraphics[width=.3\textwidth]{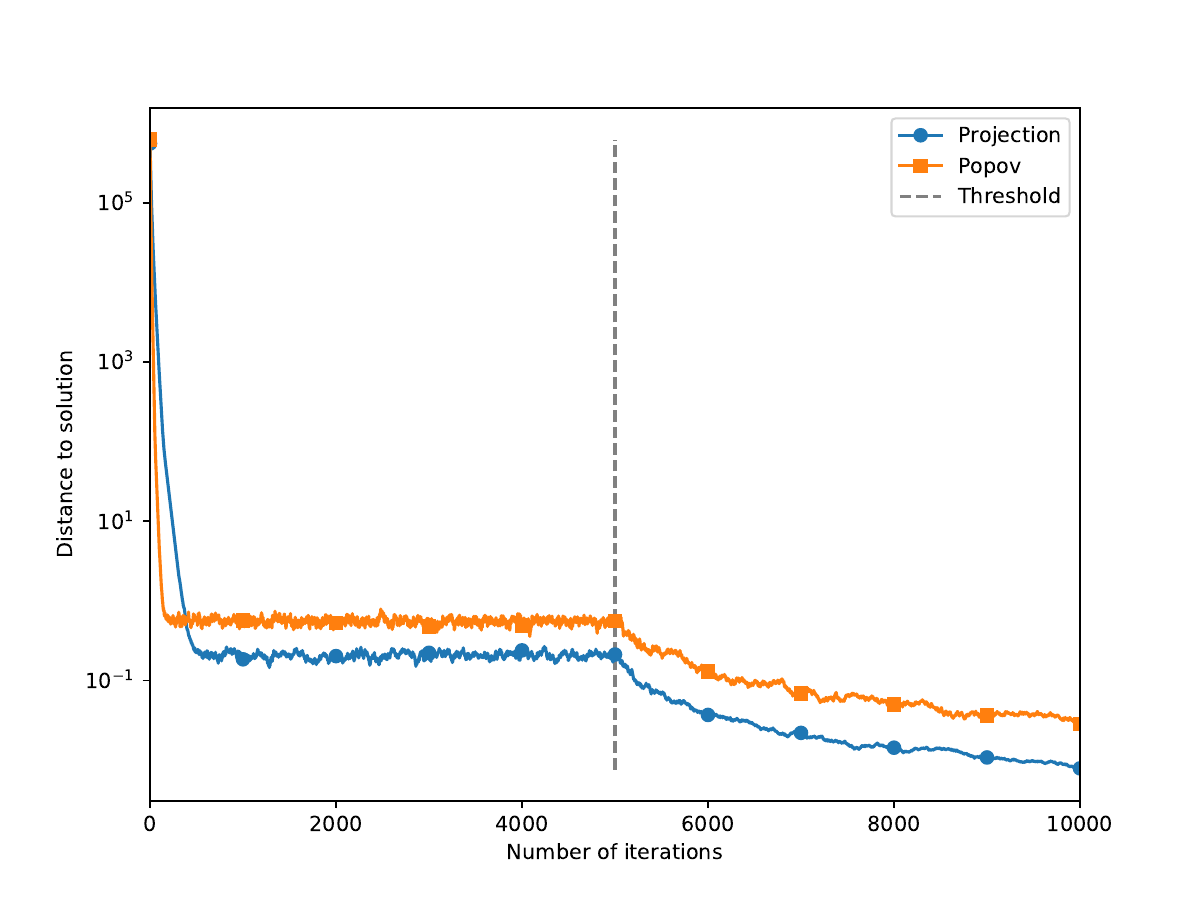}
}
\subfigure[$\kappa_F=100.6$]{
\includegraphics[width=.3\textwidth]{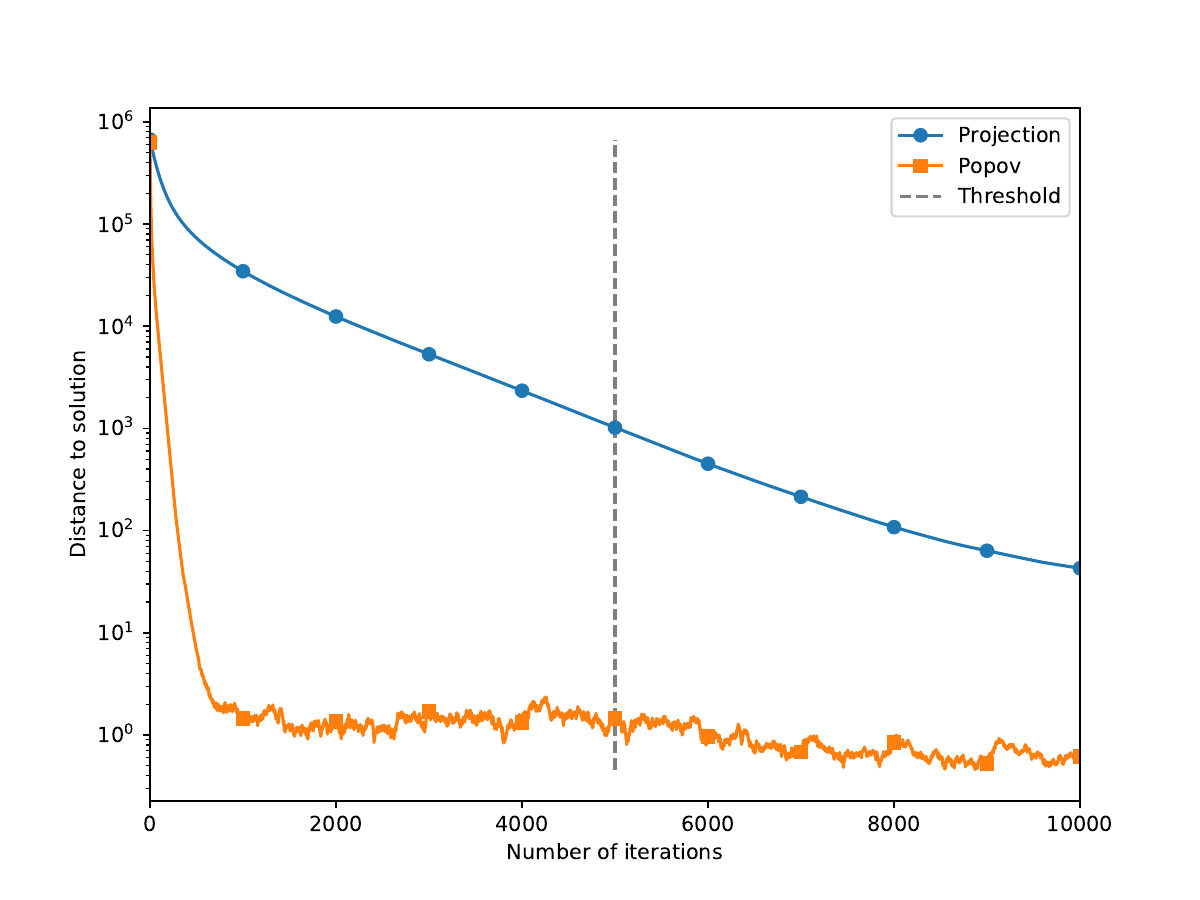}
}
\subfigure[$\kappa_F = 1007.4$]{
\includegraphics[width=.3\textwidth]{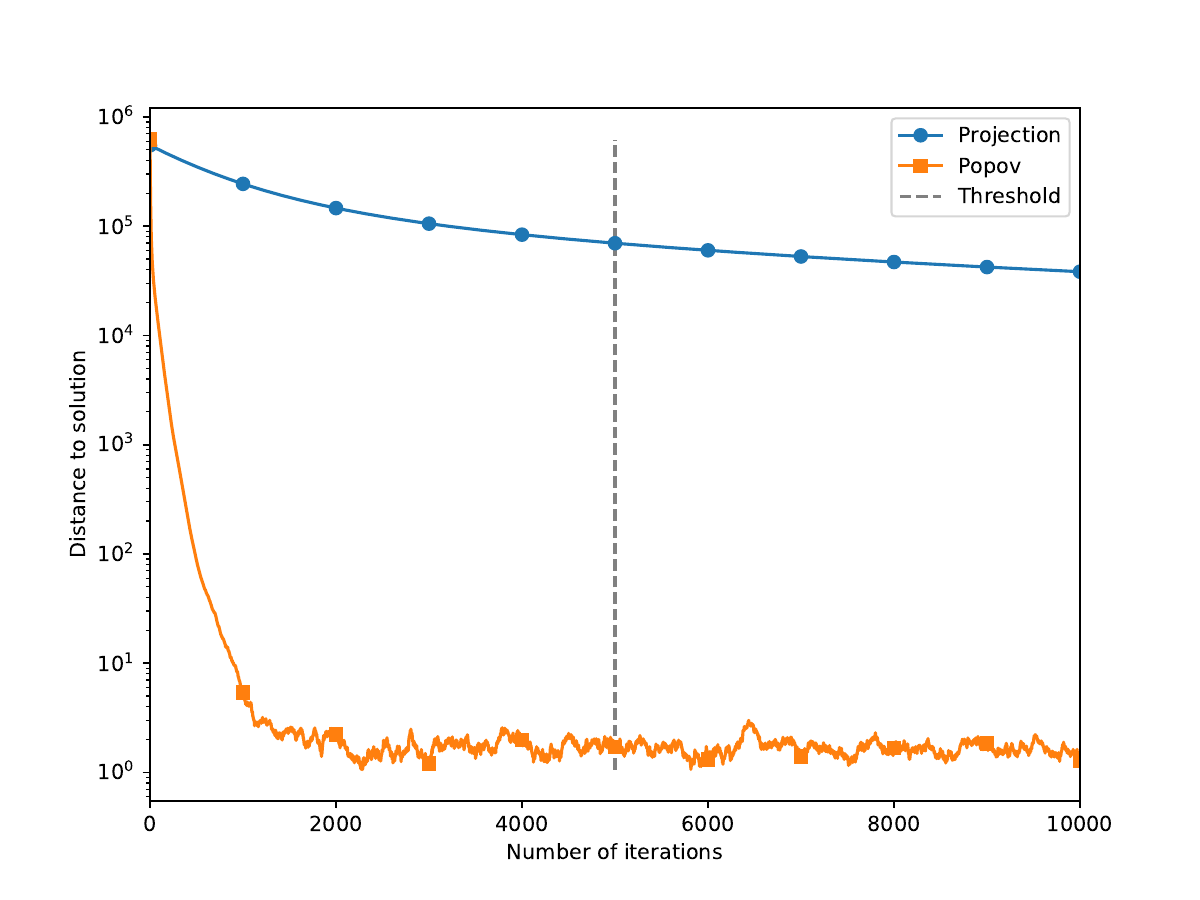}
}
\caption{Comparison of stochastic Popov method and stochastic projection method. Where $\kappa_F$ is an average (over the number of simulations) condition number of the corresponding operator. 
}
\label{fig:stichstepsizes}
\end{figure*}

\label{Numerical}
\dv{In this section, we consider the following non-monotone and discontinuous stochastic operator:}
\begin{equation}
    \label{eq-def-f}
\Phi(u, \xi) = c \left(\begin{bmatrix}
A_1 & A_2\\
-A_2' & A_3 
\end{bmatrix} u +  \begin{bmatrix}
b_1 \\
b_2
\end{bmatrix}\right) + \xi,  
\end{equation}
\dv{where 
\[c= \left\{ \begin{array}{ll} 1,  \|u\| \leq 10, \\ 0.5,  \|u\| > 10, \end{array} \right.,\]
$u = (u_1, u_2)'$ is a decisions vector with $u_1 \in \mathbb{R}^{\bd}, u_2 \in \mathbb{R}^{\bs}$, and $\xi$ is a random vector with entries being independent zero-mean Gaussian variables with variance $\sigma^2 = 1/(\bd + \bs)$. } 
We consider SVI$(U,F)$ where $U=\mathbb{R}^{m+s}$ and 
$F(u)=\mathbb{E}[\Phi(u,\xi)].$
This stochastic VI satisfies Assumption~\ref{asum-samples} when the samples $\{\xi_k\}$ are drawn independently according to the distribution of $\xi$.
Since the constraint set is $U = \mathbb{R}^{\bd+\bs}$, the SVI$(U,F)$ in~\eqref{VI} reduces to the problem of determining a point $u^*\in \mathbb{R}^{\bd + \bs}$ such that $F(u^*) = 0$, or equivalently
$\mathbb{E}[\Phi(u^*,\xi)] =0.$

In our experiments, the dimensions of the decision variables $u_1$ and $u_2$ are $\bd=\bs=30$. We generate positive definite symmetric matrices $A_1$ and $A_3$  with smallest eigenvalues $\mu_{A_1}$ and $\mu_{A_3}$, respectively. The entries of matrix $A_2$ and vectors $b_1, b_2$ are sampled from a zero mean normal distribution with variances $\sigma_{A_2}^2 = 1 / (\bd + \bs)^2$, $\sigma_{b}^2 = 1 / (\bd + \bs)$, respectively. 
 Based on the generation process, the operator $F(\cdot)$, defined in~\eqref{eq-def-f}, \dv{has linear growth and $2$-quasi sharpness property}, and so it satisfies Assumptions~\ref{asum-Linear} and~\ref{asum-sharp} (see Appendix~\ref{Numeric-Details} for the proofs). For this setting, we compare a stochastic variant of the projection method~\eqref{eq-proj-det}, where $F(u_k)$ is replaced by a sample $\Phi(u_k,\xi_k)$ and the stochastic Popov method~\eqref{eq-popov-stoch}. 
Our plots show the performance of the methods in terms of the distances $\|u_k-u^*\|^2$ between the iterates $u_k$ and the {solution} $u^*$ for both methods. In the plots, these distances are referred to as the {{\it distance to solution}}.

\dv{To obtain different condition numbers,} we set \dv{smallest eigenvalues  of $A_1$, $A_3$ to be} $\mu_{A_1} = \mu_{A_3} \in \{0.2, 0.02, 0.002\}$ and \dv{ largest eigenvalues to be $C_{A_1} = C_{A_3} = 1.0$.}
For each scenario, we run twenty simulations, and the plots
show the performance for the average trajectories and  average condition number  $\kappa_F$ (averaged over the number of simulations). 
{The stepsize for the Popov method is selected} according to Theorem~\ref{Theorem-Lipschitz-quasi-rate-exp}, while the stepsize for the projection method is chosen according to the rule given in~\eqref{eq-lemma3stitchstep}  with $d =  \mu/C^2$, $a =  \mu$ as in  \cite{loizou2021stochastic}. We choose {$K=10000$} to determine the stepsize values for both methods. The results are presented in Figure~\ref{fig:stichstepsizes}. 

As predicted in theory, stochastic Popov method gets within a neighborhood of the solution faster than stochastic projection method in all three cases. Moreover, for larger $\kappa_F$,  the number of iterations required to approach a neighborhood of the solution for the Popov method stays the same, while the projection method struggles to get to the solution. The main reason for this behavior is small step-size value for the projection method due to the condition $\alpha_k \leq \mu/C^2 =(\kappa_F C)^{-1} $. In Appendix~\ref{Numeric-Details},
 we illustrate similar results for finite-sum VIs with similar gains.
\section{Discussion}
\label{Discussion}
We have considered non-monotone SVIs under \emph{linear growth} condition on operators when the solution set is not necessarily a singleton. The class of operators with \emph{linear growth} includes Lipschitz continuous and bounded operators. Focusing on the convergence of the stochastic Popov method, we have proposed a broad class of structured non-monotone VIs called $p$-quasi sharp, which  generalizes the weak-sharpness condition for monotone VIs. We have proved the \emph{a.s.\ }last iterate convergence to a solution for the Popov method under $p$-quasi sharpness condition for all $p>0$. Among all existing results on \emph{a.s.\ } convergence of the Popov method, ours is the most extensive.
Moreover, we have proved that, for the Popov method for distinct assumptions on the operator, the last iterate converges in the distance to the solution set. Furthermore, we have obtained the optimal convergence rate of the Popov method {for Lipschitz continuous and $2$-quasi sharp operators.} This work presents interesting questions for further analysis of non-monotone $p$-quasi sharp operators. A question of independent interest is whether it is possible, for \emph{linear growth} operators, to {improve the established convergence rate by a factor of $\kappa$ and obtain the estimate in the order of} $\mathcal{O}(C  R_0 \exp [- \mu K / C] / \mu + (\sigma^2 + D^2)/ \mu^2 K)$.


\bibliographystyle{abbrvnat}
\bibliography{bib_last}

\newpage
\appendix
\onecolumn

\section{On $p$-Quasi Sharpness}
\label{appendix-A}
We provide proof that operator from the Example 1 
is $p$-quasi sharp and has linear growth. Moreover, such operator does not satisfy assumptions typically studied in the existing literature.
\begin{proof}
Firstly, we find solution set of variational inequality VI$( \mathbb{R}^2,F)$. Since $U = \mathbb{R}^2$, a solution $u^*$ of VI$(\mathbb{R}^2,F)$ must satisfy $F(u^*) =0$. Let $u^*$ be an arbitrary solution, 
then $\text{sign}(u^*_1) |u^*_1|^{p-1} + u^*_2 = 0$ and $\text{sign}(u^*_2) |u^*_2|^{p-1} - u^*_1 = 0$. From the first equality it follows that $\text{sign}(u^*_1)  = -\text{sign}(u^*_2)$, while from the second inequality it follows that $\text{sign}(u^*_2) = \text{sign}(u^*_1)$. Hence $u^*_1 = u^*_2 = 0$, and VI$(\mathbb{R}^2,F)$ has a unique solution $u^* = (0, 0 )$. 

Moreover, this operator has $p$-quasi sharpness property with $p \geq 1$ and $\mu=2^{1 - p}$. To see this, let  $\|u\| > 1$. Then, we have:
\begin{equation}
\label{example-proof-1}
\begin{aligned}
    \langle F(u), u - u^* \rangle & =   \left\langle \begin{bmatrix} \text{sign}(u_1) |u_1|^{p-1} + u_2 \\ \text{sign} (u_2) |u_2|^{p-1} - u_1 \end{bmatrix}, \begin{bmatrix} u_1 \\ u_2 \end{bmatrix} - \begin{bmatrix} 0 \\ 0 \end{bmatrix}  \right\rangle \cr
    &= |u_1|^{p} + |u_2|^{p} \cr
    &\geq 2^{1 - p} \left( |u_1| + |u_2| \right)^{p} \quad \text{Jensen inequality for  a convex function } |\cdot|^{p} \text{ since  } p \ge 1   \cr 
    &\geq 2^{1 - p} \left( \sqrt{u_1^2 + u^2_2} \right)^{p}  \quad \text{ due to $\|\cdot\|_1 \geq \|\cdot\|_2$ and monotonicity of $|\cdot|^p$} \cr
    &= 2^{1 - p} \dist^{p}(u, U^*) .
\end{aligned}
\end{equation}
In case when $\|u\| \leq 1$, the arguments are the same and we get $\langle F(u), u - u^* \rangle \geq 2^{2 - p}\dist^{p}(u, U^*)$. Moreover, it can be shown that operator $F(\cdot)$ is not monotone for $p>1$. Consider two points $u = (u_1, u_2)'$, where $u_1 = 0, u_2 = 1$, and $v = (v_1, v_2)'$, where $v_1 = 0, v_2 = 1 + \frac{1}{5 (p-1)} $. Then, $F(u) = (2, 2)'$, and $F(v) = (1 + \frac{1}{5 (p-1)}, (1 + \frac{1}{5 (p-1)})^{p - 1})'$,
and we have
\begin{align*}
\la F(u) - F(v), u - v\ra &= \la \begin{bmatrix} 1  - \frac{1}{5 (p-1)} \\ 2 - (1 + \frac{1}{5 (p-1)})^{p-1} \end{bmatrix},\begin{bmatrix} 0 \\ -\frac{1}{5 (p-1)} \end{bmatrix} \ra \cr
&= - \frac{1}{5 (p-1)} (2 - (1 + \frac{1}{5 (p-1)})^{p-1}) \cr
& \leq - \frac{1}{5 (p-1)} (2 - e^{0.2}) < 0
\end{align*}
where the inequality holds since $(1 + a / x)^{x} \leq e^{a}$. 

\dv{Next, we show that $F$ is discontinuous at $u = (0, 1)'$. Consider $v_k = (0, 1 +1/k)'$ and notice that as $k \rightarrow \infty$, $v_k \rightarrow u$, but $\lim_{k \rightarrow \infty } \|F(u) - F(v_k)\| =\sqrt{2}$. Hence, $F$ is discontinuous at $u = (0, 1)'$.
Now, we show that operator $F$ has linear growth for $p \leq 2$. In case when $\|u\| \leq 1$, $\|F(u)\| = 2 \sqrt{ (\text{sign}(u_1) |u_1|^{p-1} + u_2 )^2 + ( \text{sign}(u_2) |u_2|^{p-1} - u_1 )^2} \leq 2 \sqrt{2^2 + 2^2} = 4 \sqrt{2}$. For $\|u\| > 1$:
\begin{align}
 \|F(u)\| &= \| \begin{bmatrix} u_2 \\ -u_1 \end{bmatrix} + \begin{bmatrix}\text{sign} (u_1) |u_1|^{p-1}\\ \text{sign} (u_2) |u_2|^{p-1} \end{bmatrix} \| \cr
 &\leq \|u\| + \|\begin{bmatrix}\text{sign} (u_1) |u_1|^{p-1}\\ \text{sign} (u_2) |u_2|^{p-1} \end{bmatrix}\| \cr
&\leq \|u\| + \sqrt{(|u_1|^{p-1})^2 + (|u_2|^{p-1})^2} \cr
&\leq \|u\| + \sqrt{u_1^2 + u_2^2 + 2} \cr
&\leq 2\|u\| + \sqrt{2} .\cr
\end{align} 
Combining these two cases, we obtain that $\|F(u)\| \leq 2 \|u\| + 4 \sqrt{2}$ for all $u \in \mathbb{R}^2$.}

\dv{Finally, we show that $F$ does not satisfy quasi-strong monotonicity for any $p \in (0, 2) \cup (2, \infty)$. To arrive at contradiction, we assume that $F$ is quasi-strong monotone with $\mu > 0$. Then, for all $u \in \mathbb{R}^2$
\[ \la F(u), u - u^* \ra  \geq \mu \|u - u^*\|^2. \]
Consider $u = (u_1, 0)$. Similar to the derivation in~\eqref{example-proof-1}, we can see that
\[\la F(u), u - u^* \ra = c(|u_1|^p + |u_2|^p) = c |u_1|^p. \]
Since $\|u - u^*\|^2 = \|u\|^2 = u_1^2$, the quasi-strong monotonicity would imply that  the following inequality holds for $p>0$ and $p\ne2$, and for any $u_1 \in \mathbb{R}$,
\[c |u_1|^p \geq \mu u_1^2, \]
which is a contradiction.}
\end{proof}

Also, we present visualization of vector field of operator from Example 1 for different values $p> 0$.
\begin{figure*}[ht!]
    \subfigure[$p=1.0$]{
    \includegraphics[width=.3\textwidth]{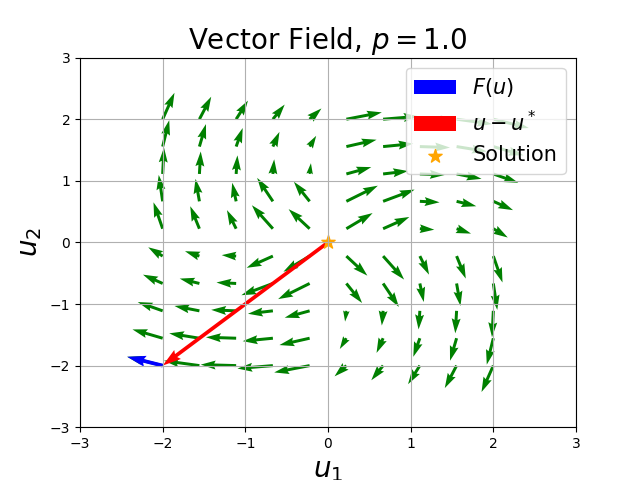}
    }
    \subfigure[$p=1.5$]{
    \includegraphics[width=.3\textwidth]{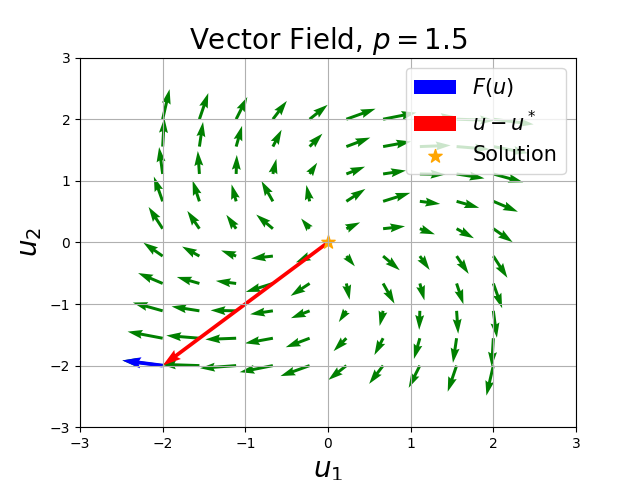}
    }
    \subfigure[$p=2.0$]{
    \includegraphics[width=.3\textwidth]{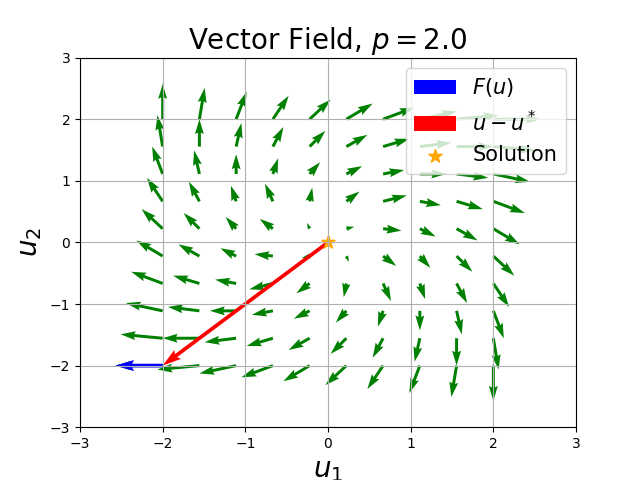}
    }
\caption{Vector field of operators from Example 1 with different  values $p \in \{1.0, 1.5, 2.0\}$}
\label{fig:field_p}
\end{figure*}

Next, we present a proof of Proposition~\ref{prop-1}
\begin{proof}
 To arrive at a contradiction, assume that $F(\cdot)$ has linear growth on $U$. Let $u \in U$ be an arbitrary point outside of the set $U^*$ and let $P_{U^*}(u)$ be a projection of $u$ on the solution set $U^*$. By the definition of $p$-quasi sharpness and the fact $\dist(u,U^*)=\|u-P_{U^*}(u)\|$, we have that 
$\mu \|u - P_{U^*}(u)\|^p \leq \langle F(u), u - u^* \rangle $ for all $u^*\in U^*$.
Thus, by letting $u^*=P_{U^*}(u)$, we obtain
\begin{align*}
\mu \|u - P_{U^*}(u)\|^p 
&\leq \|F(u)\| \|u - P_{U^*}(u)\|  \quad \text{ using Cauchy-Schwarz } \cr
&\leq (C \|u\| + D) \|u - P_{U^*}(u)\|  \quad \text{ using linear growth }\cr
&\leq (C \|u- P_{U^*}(u)\| +  \|P_{U^*}(u)\|+D) \|u - P_{U^*}(u)\| \; \text{using triangle inequality.} 
\end{align*}
Since $u\not\in U^*$, we have that $\|u - P_{U^*}(u)\| >0$
and upon dividing with $\|u - P_{U^*}(u)\|$, we find that 
\begin{align*}
\mu \|u - P_{U^*}(u)\|^{p-1} 
&\le  
C \|u- P_{U^*}(u)\| +  \|P_{U^*}(u)\|+D \le 
C \|u - P_{U^*}(u)\| + \max_{u^*\in U^*}\|u^*\| +D,
\end{align*}
which leads to a contradiction in view of 
the boundedness of the solution set $U^*$, and the facts that $p>2$ and that the set $U$ is unbounded.

\end{proof}
\section{Almost Sure and in-Expectation Convergence}
In our analysis of the Popov method~\eqref{eq-popov-stoch} we use the properties of the projection operator $P_U(\cdot)$ given in the following lemma.
\begin{lemma}\label{lem-proj} (Theorem 1.5.5 and Lemma 12.1.13 in~\cite{facchinei2003finite}) Given a convex closed set $U\subset\mathbb{R}^\bd,$ the projection operator $P_U(\cdot)$ has the following properties:
\begin{equation}
    \label{eq-proj1}
    \langle v - P_{U}(v), u - P_{U}(v) \rangle \leq 0  \quad \hbox{for all } u \in U, v \in \mathbb{R}^\bd,
\end{equation}
\begin{equation}
    \label{eq-proj3} 
    \|u - P_{U}(v)\|^2 \leq \|u - v\|^2 - \|v - P_{U}(v)\|^2 \quad \hbox{for all } u \in U, v \in \mathbb{R}^\bd,
\end{equation}
\begin{equation}
    \label{eq-proj5}
    \|P_U(u)-P_U(v)\|\le \|u-v\| \quad \hbox{for all } u , v \in \mathbb{R}^\bd.
  \end{equation}  
\end{lemma}

\subsection{Proof of Lemma~\ref{Lemma1}}

\begin{proof}
Let $k\ge 1$ be arbitrary but fixed. From the definition of $u_{k+1}$ in~\eqref{eq-popov-stoch}, we have 
$\|u_{k+1} - y\|^2 = \|P_{U} (u_{k} - \alpha_k  \Phi(h_k, \xi_k)) - y\|$ for any $y \in U$. Using the inequality~\eqref{eq-proj3} of Lemma~\ref{lem-proj} we obtain for any $y\in U$,
\begin{equation}
\begin{aligned}
\label{eq-lemmaproof1_1}
\|u_{k+1} - y\|^2 &\leq \|u_k - \alpha_k \Phi(h_{k}, \xi_k) - y\|^2 - \|u_{k} - \alpha_k \Phi(h_{k}, \xi_k) - u_{k+1}\|^2 \\
&= \|u_k  - y\|^2 - \|u_{k+1} - u_k\|^2 - 2\alpha_k \langle \Phi(h_{k}, \xi_k), u_{k+1} - y\rangle.
\end{aligned}
\end{equation}
We next consider the term $\|u_{k+1} - u_k\|^2$, where we add 
and subtract $h_k$, and thus obtain
\begin{equation}
\begin{aligned}
\label{eq-lemmaproof1_2}
\|u_{k+1} -u_k\|^2 = &\|(u_{k+1} -h_k) - (u_k - h_k)\|^2  \\
= &\|u_{k+1} -h_k\|^2 + \|u_k - h_k\|^2 - 2\langle u_k - h_k, u_{k+1} - h_k \rangle \\
=  &\|u_{k+1} -h_k\|^2 + \|u_k - h_k\|^2 - 2\langle u_k - \alpha_k \Phi(h_{k-1}, \xi_{k-1}) - h_k, u_{k+1} - h_k \rangle \\
&- 2 \alpha_k \langle \Phi(h_{k-1}, \xi_{k-1}), u_{k+1} - h_k\rangle,
\end{aligned}
\end{equation}
where the last equality is obtained by
adding and subtracting $2 \alpha_k \langle \Phi(h_{k-1}, \xi_{k-1}), u_{k+1} - h_k\rangle$. Next, we use the projection property~\eqref{eq-proj1} of Lemma~\ref{lem-proj}, where we let
$v = u_k - 2 \alpha_k \Phi(h_{k-1}, \xi_{k-1})$, $u = u_{k+1}$, and
$h_k=P_U(v)$ (which follows by the definition of $h_k$ in the method~\eqref{eq-popov-stoch}). Then, it follows that
\begin{equation}
\begin{aligned}
\langle u_k - \alpha_k \Phi(h_{k-1}, \xi_{k-1}) - h_k, u_{k+1} - h_k \rangle \leq 0.
\end{aligned}
\end{equation}
Therefore,
\begin{equation}
\begin{aligned}
\label{eq-lemmaproof1_3}
\|u_{k+1} - u_k\|^2 &\ge \|u_{k+1} -h_k\|^2 + \|u_k - h_k\|^2 - 2 \alpha_k \langle \Phi(h_{k-1}, \xi_{k-1}), u_{k+1} - h_k\rangle
\end{aligned}
\end{equation}

Combining \eqref{eq-lemmaproof1_1} and \eqref{eq-lemmaproof1_3} 
we can see that for any $y\in U$,
\begin{equation}
\begin{aligned}
\label{eq-lemmaproof1_4}
\|u_{k+1} - y\|^2 \leq &\|u_k  - y\|^2 - \|u_{k+1} -h_k\|^2 - \|u_k - h_k\|^2  - 2 \alpha_{k} \langle \Phi(h_{k}, \xi_k), u_{k+1}-h_k\rangle \cr
&\  
- 2 \alpha_{k} \langle \Phi(h_{k}, \xi_k), h_k-y\rangle
+ 2 \alpha_{k} \langle  \Phi(h_{k-1}, \xi_{k-1}), u_{k+1}-h_k\rangle\cr
= &\|u_k  - y\|^2 - \|u_{k+1} -h_k\|^2 - \|u_k - h_k\|^2
- 2 \alpha_{k} \langle \Phi(h_{k}, \xi_k), h_k-y\rangle\cr
&\ + 2 \alpha_{k} \langle \Phi(h_{k-1}, \xi_{k-1}) - \Phi(h_{k}, \xi_k), u_{k+1}-h_k\rangle.  
\end{aligned}
\end{equation}
To estimate the last inner product in~\eqref{eq-lemmaproof1_4}, we write
\[\langle \Phi(h_{k-1}, \xi_{k-1}) - \Phi(h_{k}, \xi_k), u_{k+1}-h_k\rangle \le \|\Phi(h_{k-1}, \xi_{k-1}) - \Phi(h_{k}, \xi_k)\|\, \|u_{k+1}-h_k\|. \]
From the definitions of $u_{k+1}$ and $h_{k+1}$ in~\eqref{eq-popov-stoch}, we have 
$u_{k+1} =P_{U} (u_{k} - \alpha_k  \Phi(h_k, \xi_k))$ and 
$h_k=P_U(u_k-\alpha_k\Phi(h_{k-1},\xi_{k-1}))$. 
Thus, by using the Lipschitz property of the projection operator (see relation~\eqref{eq-proj5} in Lemma~\ref{lem-proj}),
we obtain $\|u_{k+1}-h_k\|\le \alpha_k\|\Phi(h_{k-1},\xi_{k-1})-\Phi(h_k,\xi_k)\|$, implying that
\[\langle \Phi(h_{k-1}, \xi_{k-1}) - \Phi(h_{k}, \xi_k), u_{k+1}-h_k\rangle \le \alpha_k\|\Phi(h_{k-1}, \xi_{k-1}) - \Phi(h_{k}, \xi_k)\|^2.\]
Upon substituting the preceding estimate back in relation~\eqref{eq-lemmaproof1_4},
we have that
\begin{equation}
\begin{aligned}
\label{eq-lemmaproof1_5}
\|u_{k+1} - y\|^2 \leq &\|u_k  - y\|^2 - \|u_{k+1} -h_k\|^2 - \|u_k - h_k\|^2  - 2 \alpha_{k} \langle \Phi(h_{k}, \xi_k), h_k-y\rangle \\
&+ 2 \alpha_{k}^2 \|\Phi(h_{k}, \xi_k) -  \Phi(h_{k-1}, \xi_{k-1})\|^2.
\end{aligned}
\end{equation}

In the last term of~\eqref{eq-lemmaproof1_5}, we add and subtract $F(h_k)$ and $F(h_{k-1})$. Recalling that $e_{k} = \Phi(h_k, \xi_k) - F(h_k)$, we obtain
\begin{equation}
\begin{aligned}
\label{eq-lemmaproof1_6}
\|\Phi(h_{k}, &\xi_k)  -  \Phi(h_{k-1}, \xi_{k-1})\|^2 \cr 
& = \|(\Phi(h_{k}, \xi_k) -  F(h_k)) + (F(h_k) - F(h_{k-1})) + (F(h_{k-1}) -  \Phi(h_{k-1}, \xi_{k-1}))\|^2 \\
&\leq 3 \| \Phi(h_{k}, \xi_k) -  F(h_k)\|^2 + 3 \|F(h_k) - F(h_{k-1})\|^2 + 3\|F(h_{k-1}) -  \Phi(h_{k-1}, \xi_{k-1})\|^2 \\
&\leq 3 \|F(h_k) - F(h_{k-1})\|^2 + 3( \|e_k\|^2 + \|e_{k-1}\|^2),
\end{aligned}
\end{equation}
where the first inequality follows from $(\sum_{i=1}^m a_i)^2\le m\sum_{i=1}^m a_i^2$, which is valid for any scalars $a_i$, $i=1,\ldots,m,$ and any integer $m\ge 1$.
Combining relations~\eqref{eq-lemmaproof1_5} and~\eqref{eq-lemmaproof1_6}, and using $e_{k} = \Phi(h_k, \xi_k) - F(h_k)$,
we obtain the desired relation:
\begin{equation}
\begin{aligned}
\label{eq-lemmaproof1_7}
\|u_{k+1} - y\|^2 \leq &\|u_k  - y\|^2 - \|u_{k+1} -h_k\|^2 - \|u_k - h_k\|^2  - 2 \alpha_{k} \langle e_k +F(h_{k}), h_k-y\rangle \\
&+ 6 \alpha_{k}^2 \|F(h_k) -  F(h_{k-1}) \|^2 + 6 \alpha_k^2 ( \|e_k\|^2 + \|e_{k-1}\|^2).
\end{aligned}
\end{equation}
\end{proof}

\subsection{Linear Growth Condition}
In the following lemma, we refine Lemma~\ref{Lemma1} for the case when the operator $F(\cdot)$ has a linear growth
(see Assumption~\ref{asum-Linear}).
The part~(a) of the following lemma 
gives a suitable relation for our convergence rate analysis of the method~\eqref{eq-popov-stoch}, while part~(b) is used
for establishing almost sure convergence of the method.
\begin{lemma}\label{Lemma7} Assume that $U$ is a closed convex set and that the operator $F(\cdot):U\to\mathbb{R}^\bd$ has a linear growth on the set $U$. Then, the iterates $u_k$ and $h_k$ of the stochastic Popov method~\eqref{eq-popov-stoch} satisfy
the following relations:
\begin{itemize}
\item[(a)]
For all $y \in U$ and for all $k \ge 1$,
\begin{equation}
\begin{aligned}
\label{eq-lemma7-statement2}
\|u_{k+1} - y\|^2 \leq &\|u_k  - y\|^2 - \|u_{k+1} -h_k\|^2 - \|u_k - h_k\|^2  - 2 \alpha_{k} \langle e_k+F(h_{k}),  h_k-y\rangle \\
&+ 24 \alpha_{k}^2 C^2(  \|h_k\|^2 +  \|h_{k-1}\|^2) 
+ 6 \alpha_k^2 ( \|e_k\|^2 + \|e_{k-1}\|^2 + 4 D^2);
\end{aligned}
\end{equation}
\item [(b)]
For all $y\in U$, $z\in\mathbb{R}^\bd$, and for all $k \ge 1$,
\begin{equation}
\begin{aligned}
\label{eq-lemma7-statement-1}
\|u_{k+1} - y\|^2 \leq &\|u_k  - y\|^2 - \|u_{k+1} -h_k\|^2 - \|u_k - h_k\|^2  - 2 \alpha_{k} \langle e_k+F(h_{k}), h_k-y\rangle \\
&+ 72 \alpha_{k}^2 C^2(  \|h_k - u_k\|^2 +  \|u_k - h_{k-1}\|^2 + 2\|u_k - z\|^2)   \\
&+ 6 \alpha_k^2 ( \|e_k\|^2 + \|e_{k-1}\|^2 + 4 D^2 + 24 \|z\|^2);
\end{aligned}
\end{equation}
\end{itemize}
where  $e_{k} = \Phi(h_k, \xi_k) - F(h_k)$ for all $k\ge0$.
\end{lemma}

\begin{proof}
Let $k \ge 1$ and $y\in U$ be arbitrary.
By Lemma~\ref{Lemma1}, we have 
\begin{equation}
\begin{aligned}
\label{eq-lemma7-1}
\|u_{k+1} - y\|^2 \leq &\|u_k  - y\|^2 - \|u_{k+1} -h_k\|^2 - \|u_k - h_k\|^2  - 2 \alpha_{k} \langle e_k +F(h_{k}), h_k-y\rangle \\
&+ 6 \alpha_{k}^2 \|F(h_k) -  F(h_{k-1}) \|^2 + 6 \alpha_k^2 ( \|e_k\|^2 + \|e_{k-1}\|^2).
\end{aligned}
\end{equation}
Using $(\sum_{i=1}^m a_i)^2\le m\sum_{i=1}^m a_i^2$, which is valid for any scalars $a_i$, $i=1,\ldots,m,$ and any integer $m\ge 1$, to estimate $\|F(h_k) -  F(h_{k-1})\|^2$,
and the linear growth of operator $F(\cdot)$ we obtain
\begin{equation}
\begin{aligned}
\label{eq-lemma7-2}
\|F(h_k) -  F(h_{k-1}) \|^2 &\leq 2\|F(h_k)\|^2 +  2\|F(h_{k-1}) \|^2  \\
&\leq 4 C^2(\|h_k\|^2 + \|h_{k-1}\|^2) + 4 D^2.
\end{aligned}
\end{equation}
By substituting the preceding estimate back in relation~\eqref{eq-lemma7-1} we arrive at the relation in part (a).

To obtain the relation in part (b), for $\|h_k\|^2$ we write
\[\|h_k\|^2 =\|(h_k-u_k) + (u_k-z) + z\|^2.\]
where $z\in\mathbb{R}^\bd$ is arbitrary.
Using $(\sum_{i=1}^m a_i)^2\le m\sum_{i=1}^m a_i^2$, with $m=3$,
we find that
\[\|h_k\|^2 \le 3\left(\|h_k-u_k\|^2 + \|u_k-z\|^2 + \|z\|^2\right).\]
Similarly, we can see that 
\[\|h_{k-1}\|^2 \le 3\left(\|h_{k-1}-u_k\|^2 + \|u_k-z\|^2 + \|z\|^2\right).\]
Therefore,
\[
\|h_k\|^2 +|h_{k-1}\|^2
\le 3\left(\|h_k-u_k\|^2 + \|h_{k-1}-u_k\|^2+ 2\|u_k-z\|^2 + 2\|z\|^2\right).
\]
Upon substituting the preceding estimate back in
relation~\eqref{eq-lemma7-2} we obtain
\begin{equation}
\begin{aligned}
\label{eq-lemma7-2-1}
\|F(h_k) -  F(h_{k-1}) \|^2 
\le 12 C^2(\|h_k-u_k\|^2 + \|h_{k-1}-u_k\|^2+ 2\|u_k-z\|^2 + 2\|z\|^2) 
+ 4 D^2.
\end{aligned}
\end{equation}
Combining the estimate in~\eqref{eq-lemma7-2-1} with relation~\eqref{eq-lemma7-1} we obtain the following relation
\begin{equation}
\begin{aligned}
\label{eq-lemma7-3}
\|u_{k+1} - y\|^2 \leq &\|u_k  - y\|^2 - \|u_{k+1} -h_k\|^2 - \|u_k - h_k\|^2  -2 \alpha_{k} \langle e_k+F(h_{k}), h_k-y\rangle \\
&+ 72 \alpha_{k}^2 C^2(  \|h_k - u_k\|^2 +  \|u_k - h_{k-1}\|^2 + 2\|u_k - z\|^2)   \\
&+ 6 \alpha_k^2 ( \|e_k\|^2 + \|e_{k-1}\|^2 + 4 D^2 + 24 \|z\|^2),
\end{aligned}
\end{equation}
which is the relation stated in part (b).
\end{proof}

In the forthcoming analysis, we use Lemma~11 \cite{polyak1987introduction}, which is stated below.
\begin{lemma} \label{lemma-polyak11} [Lemma~11 \cite{polyak1987introduction}]
Let $\{v_k\}, \{z_k\}, \{a_k\}, \{b_k\}$ be nonnegative random scalar sequences such that almost surely for all $k\ge0$,
\begin{equation}
\begin{aligned}
\label{eq-polyak-0}
\mathbb{E}[v_{k+1}\mid {\cal F}_k] \leq &(1 + a_k)v_k -z_k + b_k,
\end{aligned}
\end{equation}
where 
${\cal F}_k = \{v_0, \ldots, v_k, z_0, \ldots, z_k, a_0, \ldots, a_k,b_0, \ldots,b_k\}$, and
\emph{a.s.} $\sum_{k=0}^{\infty} a_k < \infty$, $\sum_{k=0}^{\infty} b_k < \infty$. Then, almost surely, $\lim_{k\to\infty} v_k =v $  for some nonnegative random variable $v$ and $\sum_{k=0}^{\infty} z_k < \infty$.
\end{lemma}

As a direct consequence of Lemma~\ref{lemma-polyak11}, when
the sequences $\{v_k\}, \{z_k\}, \{a_k\}, \{b_k\}$ are deterministic, we obtain the following result.
\begin{lemma} \label{lemma-polyak11-det} 
Let $\{\bar v_k\}, \{\bar z_k\}, \{\bar a_k\}, \{\bar b_k\}$ be nonnegative scalar sequences such that for all $k\ge0$,
\begin{equation}
\begin{aligned}
\label{eq-polyak-1}
\bar v_{k+1}\leq &(1 + \bar a_k)\bar v_k -\bar z_k + \bar b_k,
\end{aligned}
\end{equation}
where $\sum_{k=0}^{\infty} \bar a_k < \infty$ and $\sum_{k=0}^{\infty} \bar b_k < \infty$. Then, 
$\lim_{k\to\infty} \bar v_k = \bar v$ 
for some scalar $\bar v\ge0$ and $\sum_{k=0}^{\infty} \bar z_k < \infty$.
\end{lemma}

We also use Lebesgue Dominated Convergence Theorem, which is stated below and can be found, for example, in the textbook~\cite{billingsley}, Theorem 16.4, page 209.
\begin{theorem}[Lebesgue Dominated Convergence Theorem] 
\label{thm-lebesque}
Let $\{f_k\}$ be a sequence of functions and $g$ be a function in some measure space with a measure $\nu$, and let $|f_k|\le g$ almost everywhere. If $g$ is integrable and $f_k\to f$ almost everywhere, then $\int f_k d\nu\to\int fd\nu$.
\end{theorem}

\subsection{Proof of Theorem~\ref{Theorem-Linear-sharp-ASC}}\label{sec-proof-Theorem-Linear-sharp-ASC}
We use Lemmas~\ref{lemma-polyak11} and~\ref{lemma-polyak11-det}  to establish parts~(a) and~(b), respectively, while we use Theorem~\ref{thm-lebesque} to prove part~(c).
(a)\ 
Using Lemma~\ref{Lemma7}(b), where we set
$y = z = u^*$ for an arbitrary $u^* \in U^*$, after re-arranging the terms,
we obtain for all $u^*\in U^*$ and for all $k\ge 1$,
\begin{equation}
\begin{aligned}
\label{eq-th-linear-asc-1}
\|u_{k+1} - u^*\|^2 +& \|u_{k+1} - h_k\|^2\leq (1 + 144 \alpha_k^2 C^2)\|u_k  - u^*\|^2  -  (1 - 72 \alpha_k^2 C^2) \|u_k - h_k\|^2  \\
&- 2 \alpha_k \langle  F(h_{k}), h_k-u^* \rangle 
+ 2 \alpha_k \langle  e_k, u^* - h_k\rangle
+  72 \alpha_{k}^2 C^2 \|u_k - h_{k-1}\|^2 \cr
& + 6 \alpha_{k}^2 (\|  e_{k-1}\|^2 + \| e_{k}\|^2   
+ 24 \|u^*\|^2 + 4  D^2).
\end{aligned}
\end{equation}

Under Assumption~\ref{asum-sharp}, the following relation is valid
 for all $k\ge0$,
\begin{equation}
\begin{aligned}
\label{eq-th-linear-asc-sharp1}
\langle   F(h_k), h_k - u^*  \rangle \ge \mu \dist^p (h_k, U^*),
\end{aligned}
\end{equation}
with $p>0$ and $\mu>0$.
Combining~\eqref{eq-th-linear-asc-sharp1} with~\eqref{eq-th-linear-asc-1} we get
for all $u^*\in U^*$ and for all $k\ge 1$,
\begin{equation}
\begin{aligned}
\label{eq-th-linear-asc-2}
\|u_{k+1} - u^*\|^2 + \|u_{k+1} - h_k\|^2\leq & (1 + 144 \alpha_k^2 C^2)\|u_k  - u^*\|^2 - 2 \alpha_k \mu\dist^p (h_k, U^*)\\
-&  (1 - 72 \alpha_k^2 C^2) \|u_k - h_k\|^2 
+ 2 \alpha_k \langle  e_k, u^* - h_k\rangle \\
+ &  72 \alpha_{k}^2 C^2 \|u_k - h_{k-1}\|^2 + 6 \alpha_{k}^2 (\|  e_{k-1}\|^2 + \| e_{k}\|^2) \\   
+& 6 \alpha_{k}^2 (  24 \|u^*\|^2 + 4  D^2)  .
\end{aligned}
\end{equation}
By writing
\[72\a_k^2 C^2\|u_{k}-h_{k-1}\|^2
\le (1+144\a_k^2C^2) \|u_{k}-h_{k-1}\|^2
-\|u_{k}-h_{k-1}\|^2,\]
and regrouping some of the terms in~\eqref{eq-th-linear-asc-2}, we have 
for all $u^*\in U^*$ and for all $k\ge 1$,
\begin{equation*}
\begin{aligned}
\|u_{k+1} - u^*\|^2 + \|u_{k+1} - h_k\|^2\leq & 
(1 + 144 \alpha_k^2 C^2)(\|u_k  - u^*\|^2 +\|u_{k}-h_{k-1}\|^2) \\
-& 2 \alpha_k \mu\dist^p (h_k, U^*) -  (1 - 72 \alpha_k^2 C^2) \|u_k - h_k\|^2 \\ 
+& 2 \alpha_k \langle  e_k, u^* - h_k\rangle
-\|u_k - h_{k-1}\|^2 \\
+ & 6 \alpha_{k}^2 (\|  e_{k-1}\|^2 + \| e_{k}\|^2   
+ 24 \|u^*\|^2 + 4  D^2) .
\end{aligned}
\end{equation*}
Next, we add $7\alpha_{k+1}^2\|e_k\|^2$ to both sides of the preceding relation and,  after slightly re-arranging the terms, we  obtain for all $u^*\in U^*$ and for all $k\ge 1$,
\begin{equation}
\begin{aligned}
\label{eq-th-linear-asc-2-1}
\|u_{k+1} - u^*\|^2 +& \|u_{k+1} - h_k\|^2+ 7\alpha_{k+1}^2\|e_k\|^2 \cr
\leq & 
(1 + 144 \alpha_k^2 C^2)(\|u_k  - u^*\|^2 +\|u_{k}-h_{k-1}\|^2) +6 \alpha_{k}^2 \|  e_{k-1}\|^2 \cr
&- 2 \alpha_k \mu\dist^p (h_k, U^*) - (1 - 72 \alpha_k^2 C^2) \|u_k - h_k\|^2 \cr
&+ 2 \alpha_k \langle  e_k, u^* - h_k\rangle
-\|u_k - h_{k-1}\|^2 \\
&+ 7\alpha_{k+1}^2\|e_k\|^2+6 \alpha_{k}^2 (\| e_{k}\|^2   
+ 24 \|u^*\|^2 + 4  D^2) .
\end{aligned}
\end{equation}
We next consider the term $6 \alpha_{k}^2 \|  e_{k-1}\|^2 $
for which we write
\[6 \alpha_{k}^2 \|  e_{k-1}\|^2 =7 \alpha_{k}^2 \|  e_{k-1}\|^2 -\alpha_{k}^2 \|  e_{k-1}\|^2 
\le 7 (1 + 144 \alpha_k^2 C^2) \alpha_{k}^2 \|  e_{k-1}\|^2 -\alpha_{k}^2 \|  e_{k-1}\|^2.\]
Upon substituting the preceding estimate back in~\eqref{eq-th-linear-asc-2-1} we have that 
for all $u^*\in U^*$ and for all $k\ge 1$,
\begin{equation}
\begin{aligned}
\label{eq-th-linear-asc-2-11}
\|u_{k+1} - u^*\|^2 + & \|u_{k+1} - h_k\|^2+ 7\alpha_{k+1}^2\|e_k\|^2 \cr
\leq & 
(1 + 144 \alpha_k^2 C^2)(\|u_k  - u^*\|^2 +\|u_{k}-h_{k-1}\|^2 +7 \alpha_{k}^2 \|  e_{k-1}\|^2) \cr
&- \alpha_{k}^2 \|  e_{k-1}\|^2 
- 2 \alpha_k \mu\dist^p (h_k, U^*)
-  (1 - 72 \alpha_k^2 C^2) \|u_k - h_k\|^2 \cr
&+ 2 \alpha_k \langle  e_k, u^* - h_k\rangle
-\|u_k - h_{k-1}\|^2 \\
&+ 7\alpha_{k+1}^2\|e_k\|^2+6 \alpha_{k}^2 (\| e_{k}\|^2   
+ 24 \|u^*\|^2 + 4  D^2) .
\end{aligned}
\end{equation}

Since $\sum_{k=0}^{\infty} \alpha_k^2 < \infty$, it follows that $\alpha_k \to 0$, so there exists an index $N\ge 1$ such that the stepsize satisfies $1-72\a_k^2C^2\ge 1/2$ for all $k \ge N$. Thus, by defining
\begin{equation}
\label{eq-def-vk}
 v_k=\|u_k  - u^*\|^2+\|u_k-h_{k-1}\|^2 +6\a_k^2\|e_{k-1}\|^2\qquad\hbox{for all }k\ge1,
\end{equation}
from relation~\eqref{eq-th-linear-asc-2-11}
we obtain for all $u^*\in U^*$ and $k\ge N$,
\begin{equation}
\begin{aligned}
\label{eq-th-linear-asc-2-2}
v_{k+1}
\leq & 
(1 + 144 \alpha_k^2 C^2)v_k - \alpha_{k}^2 \|  e_{k-1}\|^2 
- 2 \alpha_k \mu\dist^p (h_k, U^*)
-  \frac{1}{2}\|u_k - h_k\|^2 \cr
+& 2 \alpha_k \langle  e_k, u^* - h_k\rangle
-\|u_k - h_{k-1}\|^2 + 7\alpha_{k+1}^2\|e_k\|^2+6 \alpha_{k}^2 (\| e_{k}\|^2   
+ 24 \|u^*\|^2 + 4  D^2) .
\end{aligned}
\end{equation}

Recalling that $e_k=\Phi(h_k,\xi_k)-F(h_k)$ and  
using the stochastic properties of $\xi_k$ imposed by Assumption~\ref{asum-samples}, we have
$\mathbb{E}[ \langle  e_k, h_k-u^* \rangle|\mathcal{F}_{k-1}] = 0$ and $\mathbb{E}[\|e_k\|^2|\mathcal{F}_{k-1}] \le \sigma^2$ for all $k\ge1$. Thus, by
taking the conditional expectation on $\mathcal{F}_{k-1}=\{\xi_0,\ldots,\xi_{k-1}\}$
in relation~\eqref{eq-th-linear-asc-2-2}, we obtain 
for all $u^*\in U^*$ and for all $k\ge N$,
\begin{equation}
\begin{aligned}
\label{eq-th-linear-asc-3}
\mathbb{E}[v_{k+1}\mid  \mathcal{F}_{k-1}]
\leq & (1 + 144 \alpha_k^2 C^2) v_k  
- \alpha_{k}^2 \|  e_{k-1}\|^2
-2 \alpha_k \mu  \dist^p(h_k, U^*) 
-  \frac{1}{2} \|u_k - h_k\|^2 \cr
- & \|u_{k}-h_{k-1}\|^2 
+ 7\alpha_{k+1}^2\sigma^2
+6 \alpha_{k}^2 (\sigma^2
+ 24 \|u^*\|^2 + 4  D^2).
\end{aligned}
\end{equation}
Notice that when $u^*\in U^*$ is a fixed solution, then $\|u^*\|$ is a constant.

Since $\sum_{k=0}^{\infty} \alpha_k^2 < \infty$, the inequality in \eqref{eq-th-linear-asc-3} satisfies the conditions of Lemma~\ref{lemma-polyak11}
for all $k\ge N$, with
\[
z_k=  \alpha_{k}^2 \|  e_{k-1}\|^2+2 \alpha_k \mu  \dist^p(h_k, U^*) 
+\frac{1}{2}\|u_k - h_k\|^2 + \|u_k-h_{k-1}\|^2,\]
\[
a_k=144\a_k^2C^2,\qquad
b_k=7\alpha_{k+1}^2\sigma^2
+6 \alpha_{k}^2 (\sigma^2
+ 24 \|u^*\|^2 + 4  D^2).\]
By Lemma~\ref{lemma-polyak11} (where we shift the indices to start with $k=N$), it follows that the sequence $\{v_k\}$ converges {\it a.s.}\ to a non-negative scalar for any $u^*\in U^*$, and  almost surely we have
\[\sum_{k=N}^\infty \alpha_{k}^2 \|  e_{k-1}\|^2<\infty,
\quad
\sum_{k=N}^{\infty} \alpha_k \dist^p(h_k, U^*) < \infty,\quad \sum_{k=N}^{\infty} (\|u_k - h_k\|^2 +\|u_k-h_{k-1}\|^2)< \infty.\]
Thus, it follows that 
\begin{equation}\label{eq-an0}
\lim_{k\to\infty}\alpha_{k}^2 \|  e_{k-1}\|^2 = 0\qquad
{\it a.s.}\end{equation}
\begin{equation}\label{eq-an1}
\lim_{k\to\infty}\|u_k - h_k\| = 0\qquad
{\it a.s.}\end{equation}
\begin{equation}\label{eq-an4}
\lim_{k\to\infty}\|u_k-h_{k-1}\| =0
\qquad {\it a.s.}\end{equation}
Moreover, since $\sum_{k=0}^{\infty} \alpha_k = \infty$ , it follows that 
\[\liminf_{k\to\infty}\dist^p(h_k, U^*)=0 \quad {\it a.s.}\]

Since the sequence $\{v_k\}$ 
converges {\it a.s.}\ 
for any given $u^*\in U^*$, in view of the definition of $v_k$ in~\eqref{eq-def-vk} combined with relations\eqref{eq-an0} and~\eqref{eq-an4}, it follows that 
the sequence $\{\|u_k - u^*\|^2\}$
converges {\it a.s.}\ 
for all $u^*\in U^*$. 
Thus, the sequence $\{u_k\}$ is bounded {\it a.s.} and, consequently,
it has accumulation points {\it a.s.}
In view of relation~\eqref{eq-an1},
the sequences $\{u_k\}$ and $\{h_k\}$ have the same accumulation points.

Now, let $\{k_i\mid i\ge 1\}$ be a (random) index sequence such that 
\begin{equation}\label{eq-an2}
\lim_{i\to\infty} \dist^p(h_{k_i}, U^*)=\liminf_{k\to\infty}\dist^p(h_k, U^*)=0 \quad {\it a.s.}\end{equation}
Without loss of generality we may assume that $\{u_{k_i}\}$ is a convergent sequence (for otherwise we will select a convergent subsequence), and let $\bar u$ be its (random) limit point, i.e.,
\begin{equation}\label{eq-an3}
\lim_{i\to\infty} \|u_{k_i}-\bar u\|=0
\qquad{\it a.s.}\end{equation}
By relation~\eqref{eq-an1}, it follows that
\[\lim_{i\to\infty} \|h_{k_i}-\bar u\|=0
\qquad{\it a.s.}\]
By continuity of the distance function $\dist(\cdot,U^*)$, from relation~\eqref{eq-an2} we conclude that $\dist(\bar u,U^*)=0$ {\it a.s.}, which implies that $\bar u\in U^*$ almost surely since the set $U^*$ is closed.
 Since the sequence $\{\|u_k - u^*\|^2\}$ 
converges {\it a.s.}\ 
for any $u^*\in U^*$, it follows that 
$\{\|u_k - \bar u\|^2\}$ 
converges {\it a.s.}, and by relation~\eqref{eq-an3} we conclude that 
$\lim_{k\to\infty}\|u_k - \bar u\|^2=0$.

(b)\ Taking the total expectation in~\eqref{eq-th-linear-asc-3}, we obtain for all $u^*\in U^*$ and all $k\ge N$,
\begin{equation}
\begin{aligned}
\label{eq-th-linear-asc-4}
\mathbb{E}[v_{k+1}]
\leq & (1 + 144 \alpha_k^2 C^2) \mathbb{E}[v_k ] 
- \alpha_{k}^2 \mathbb{E}[\|  e_{k-1}\|^2]
-2 \alpha_k \mu \mathbb{E} [\dist^p(h_k, U^*) ]
-  \frac{1}{2} \mathbb{E}[\|u_k - h_k\|^2]
\cr
&- \mathbb{E}[\|u_{k}-h_{k-1}\|^2] + 7\alpha_{k+1}^2\sigma^2
+6 \alpha_{k}^2 (\sigma^2
+ 24 \|u^*\|^2 + 4  D^2).
\end{aligned}
\end{equation}
We can now apply Lemma~\ref{lemma-polyak11-det} for $k\ge N$ (instead of $k\ge0$),
with 
\[\bar v_k=\mathbb{E}[v_k ], \quad
\bar z_k=  
\alpha_{k}^2 \mathbb{E}[\|  e_{k-1}\|^2]
+2 \alpha_k \mu \mathbb{E} [\dist^p(h_k, U^*) ]
+ \frac{1}{2} \mathbb{E}[\|u_k - h_k\|^2]
+\mathbb{E}[\|u_{k}-h_{k-1}\|^2],\]
\[
\bar a_k=144\a_k^2C^2,\qquad
\bar b_k=7\alpha_{k+1}^2\sigma^2
+6 \alpha_{k}^2 (\sigma^2
+ 24 \|u^*\|^2 + 4  D^2).\]
Since $\sum_{k=0}^{\infty} \alpha_k^2 < \infty$, the inequality~\eqref{eq-th-linear-asc-4} satisfies the conditions of Lemma~\ref{lemma-polyak11-det}
for all $k\ge N$.
By Lemma~\ref{lemma-polyak11-det} (where the indices are shifted to start with $k=N$ instead of $k=0$), and the definitions of $\bar v_k$, $v_k$ in~\eqref{eq-def-vk}, and $\bar z_k$, it follows that
\begin{equation}\label{eq-exist-an}
\lim_{k\to\infty} \mathbb{E}[\|u_k  - u^*\|^2+\|u_k-h_{k-1}\|^2 +6\a_k^2\|e_{k-1}\|^2]\quad\hbox{exist for every $u^*\in U^*$},\end{equation}
\[\sum_{k=N}^\infty ( 
\alpha_{k}^2 \mathbb{E}[\|  e_{k-1}\|^2]
+2 \alpha_k \mu \mathbb{E} [\dist^p(h_k, U^*) ]
+ \frac{1}{2} \mathbb{E}[\|u_k - h_k\|^2]
+\mathbb{E}[\|u_{k}-h_{k-1}\|^2])<\infty.\]
Therefore, it follows that
\begin{equation}\label{eq-exp0-an}
\lim_{k\to\infty} (\alpha_{k}^2 \mathbb{E}[\|  e_{k-1}\|^2]+\mathbb{E}[\|u_k - h_{k-1}\|^2)=0,\end{equation}
\begin{equation}\label{eq-exp00-an}
\lim_{k\to\infty}  \mathbb{E}[\|u_k - h_k\|^2]=0.\end{equation}
From relations~\eqref{eq-exist-an} and~\eqref{eq-exp0-an} we conclude that 
\begin{equation}\label{eq-exist1-an}
\lim_{k\to\infty} \mathbb{E}[\|u_k  - u^*\|^2]\quad\hbox{exist for every $u^*\in U^*$},\end{equation}
which implies that $\{\mathbb{E}[\|u_k  - u^*\|^2]\}$ is bounded. Hence, for a fixed $u^*\in U^*$ and all $k\ge0$,
\[\mathbb{E}[\|u_k\|^2]= \mathbb{E}[\|(u_k-u^*)+u^*\|^2] 
\le \mathbb{E}[(\|u_k-u^*\|+\|u^*\|)^2
\le 2\mathbb{E}[\|u_k-u^*\|^2]+2\|u^*\|^2,\]
implying that the sequence $\{\mathbb{E}[\|u_k\|^2]\}$ is bounded. Moreover, we have that 
all $k\ge0$,
\[\mathbb{E}[\|h_k\|^2]= \mathbb{E}[\|(h_k-u_k)+u_k\|^2] 
\le \mathbb{E}[(\|h_k-u_k\|+\|u_k\|)^2
\le 2\mathbb{E}[\|h_k-u_k\|^2]+2\mathbb{E}[\|u_k\|^2],\]
thus implying that the sequence $\{\mathbb{E}[\|h_k\|^2]\}$
is bounded due to relation~\eqref{eq-exp00-an} and 
the boundedness of $\{\mathbb{E}[\|u_k\|^2]\}$.

(c)\ By part~(a), we have that almost surely
\[\lim_{k\to\infty}\|u_k-\bar u\|^2=0,\qquad\lim_{k\to\infty}\|h_k-\bar u\|^2=0,\]
for some random solution $\bar u\in U^*$.
When the set $U^*$ is bounded, we further have that
\[\|u_k-\bar u\|^2 \le (\|u_k\|+\|\bar u\|)^2 
\le 2\|u_k\|^2+ 2\|u^*\|^2
\le 2\|u_k\|^2+ 2M_0^2,\]
where $M_0=\max_{u^*\in U^*}\|u^*\|$. 
Similarly, we have
\[\|h_k-\bar u\|^2 \le 2\|h_k\|^2+ 2M_0^2.\]
We note that by part~(b), the sequences 
$\{\mathbb{E}[\|u_k\|^2]\}$ and $\{\mathbb{E}[\|h_k\|^2]\}$ are bounded. By applying the Lebesgue Dominated Convergence Theorem, with
$f_k=\|u_k-\bar u\|^2$ and $g=2\|u_k\|^2+ 2M_0^2$,
we conclude that
\[\lim_{k\to\infty}\mathbb{E}[\|u_k-\bar u\|^2]=0.\]
Similarly, applying the Lebesgue Dominated Convergence Theorem, with
$f_k=\|h_k-\bar u\|^2$ and $g=2\|h_k\|^2+ 2M_0^2$,
we obtain that
\[\lim_{k\to\infty}\mathbb{E}[\|h_k-\bar u\|^2]=0.\]

\def\a{{\alpha}}
\def\g{{\gamma}}
\section{Convergence Rates}
\subsection{Auxiliary Results}
In our analysis we make use of Lemma~3 and Lemma~7 from~\cite{stich2019unified}, as well as the sequences provided in the proofs in~\cite{stich2019unified}. 

\begin{lemma}
    \label{Lemma7-stich}
    Let $\{r_k\}$ and $\{s_k\}$ be nonnegative scalar sequences that satisfy the following relation
    \[r_{k+1}\le (1-a \gamma_k) r_k -b \gamma_k s_k + c \g_k^2\qquad\hbox{for all } k\ge0,\]
    where $a>0$, $b>0$, $c\ge0$, and 
    \[\g_k = \frac{2}{a\left( \frac{2d}{a} +k\right)} \qquad\hbox{for all }k\ge 0,\]
    where $d\ge a$.
    Then, for any given $K\ge0$, the following relation holds:
\[\frac{b}{W_K} \sum^{K}_{k=0} w_k s_k + a r_{K+1} \leq  \frac{8d^2}{a K^2}\,r_0 + \frac{2c}{aK}, \]
where $w_k=2d/a +k$,  $0\le k\le K$, and $W_K=\sum_{k=0}^K w_k$.
\end{lemma}

\begin{lemma}
    \label{Lemm3-stich}
    Let $\{r_k\}$, $\{s_k\}$, and $\{\gamma_k\}$ be nonnegative scalar sequences that satisfy the following relation
    \[r_{k+1}\le (1-a \gamma_k) r_k -b \gamma_k s_k + c \g_k^2\qquad\hbox{for all } k\ge0,\]
    where $a>0$, $b>0$, $c\ge0$, and $\g_k\le d^{-1}$ for some $d\ge a$ and for all $k\ge0$.
    Then, for any given $K\ge0$, we can choose the stapsizes $\g_k$ and the weights $w_k\ge0$, $0\le k\le K$, such that the following relation holds:
\[\frac{b}{W_K} \sum^{K}_{k=0} w_k s_k + a r_{K+1} \leq  32 d r_0 e^{-\frac{a K}{2d}} + \frac{36c}{aK}, \]
where $W_K=\sum_{k=0}^K w_k$.
\end{lemma}
A specific choice of the stepsize and the weights for which the preceding lemma holds is as follows:
\begin{eqnarray}
\label{eq-stich-step}
&\gamma_k = \frac{1}{d},\qquad w_k=\left(1-\frac{a}{d}\right)^{-(k+1)}\qquad  \text{if } K \leq \frac{d}{a},\cr
& \gamma_k= \frac{1}{d},\qquad w_k=0\qquad \text{if } K > \frac{d}{a} \ \mbox{ and } \ k < k_0,\cr
&\gamma_k = \frac{2}{a\left(\frac{2d}{a} +k- k_0\right)},\qquad
w_k=\left(\frac{2d}{a} +k- k_0\right)^2\qquad \text{if } K > \frac{d}{a} \ \mbox{ and } \ k \geq k_0,
\end{eqnarray}
where $k_0=\left\lceil \frac{K}{2}\right\rceil$.

\subsection{Proof of Theorem~\ref{Theorem-Linear-quasi-rate-sub}}\label{sec-proofTheorem-diminishing-rate}

\begin{proof}
By Lemma~\ref{Lemma7}(a), the following relation holds for all $u^*\in U^*$ and for all $k\ge1$,
\begin{equation}
\begin{aligned}
\label{eq-th-linear-quasi-rate-base}
\|u_{k+1} - u^*\|^2 \leq &\|u_k  - u^*\|^2 - \|u_{k+1} -h_k\|^2 - \|u_k - h_k\|^2  - 2 \alpha_{k} \langle e_k+F(h_{k}),  h_k-u^*\rangle \\
&+ 24 \alpha_{k}^2 C^2(  \|h_k\|^2 +  \|h_{k-1}\|^2) 
+ 6 \alpha_k^2 ( \|e_k\|^2 + \|e_{k-1}\|^2 + 4 D^2).
\end{aligned}
\end{equation}
Since the operator $F(\cdot)$ is $2$-quasi sharp, we have that
\begin{equation}
\begin{aligned}
\label{eq-th-linear-quasi-rate-sharp}
\langle F(h_{k}),  h_k-u_k^*\rangle \geq \mu \dist^2(h_k, U^*),
\end{aligned}
\end{equation}
with $\mu > 0$. From relations~\eqref{eq-th-linear-quasi-rate-base} (where we drop the term $\|u_{k+1} -h_k\|^2$) and~\eqref{eq-th-linear-quasi-rate-sharp} it follows that for all 
$u^*\in U^*$ and $k\ge1$,
\begin{equation}
\begin{aligned}
\label{eq-th-linear-quasi-rate-00}
\|u_{k+1} - u^*\|^2 \leq &\|u_k  - u^*\|^2 
- \|u_k - h_k\|^2  - 2 \alpha_{k} \langle e_k,  h_k-u^*\rangle -2\alpha_k\mu\dist^2(h_k,U^*)\\
&+ 24 \alpha_{k}^2 C^2(  \|h_k\|^2 +  \|h_{k-1}\|^2) 
+ 6 \alpha_k^2 ( \|e_k\|^2 + \|e_{k-1}\|^2 + 4 D^2).
\end{aligned}
\end{equation}
Since the solution set $U^*$ is closed, there exists a projection $u_k^*$ of the iterate $u_k$ on the set $U^*$, i.e., there exists a point $u_k^*\in U^*$ such that $\|u_k- u_k^*\|= \dist(u_k, U^*)$. Thus, by letting $u^*=u_k^*$, we obtain for all $k\ge 1$,
\begin{equation}
\begin{aligned}
\label{eq-th-linear-quasi-rate-001}
\|u_{k+1} - u_k^*\|^2 \leq &\|u_k  - u_k^*\|^2  - \|u_k - h_k\|^2  - 2 \alpha_{k} \langle e_k,  h_k-u_k^*\rangle -2\alpha_k\mu\dist^2(h_k,U^*)\\
&+ 24 \alpha_{k}^2 C^2(  \|h_k\|^2 +  \|h_{k-1}\|^2) 
+ 6 \alpha_k^2 ( \|e_k\|^2 + \|e_{k-1}\|^2 + 4 D^2).
\end{aligned}
\end{equation}
In view of $\|u_k- u_k^*\|= \dist(u_k, U^*)$ and 
$\dist(u_{k+1},U^*)\le \|u_{k+1} - u_k^*\|$, it follows that for all $k\ge1$,
\begin{equation}
\begin{aligned}
\label{eq-th-linear-quasi-rate-002}
\dist^2(u_{k+1},U^*) \leq &\dist^2(u_k,U^*) - \|u_k - h_k\|^2  - 2 \alpha_{k} \langle e_k,  h_k-u_k^*\rangle -2\alpha_k\mu\dist^2(h_k,U^*)\\
&+ 24 \alpha_{k}^2 C^2(  \|h_k\|^2 +  \|h_{k-1}\|^2) 
+ 6 \alpha_k^2 ( \|e_k\|^2 + \|e_{k-1}\|^2 + 4 D^2).
\end{aligned}
\end{equation}
By Assumption ~\ref{asum-samples}, we have that $\mathbb{E}[\|e_{k}\|^2\mid h_k] \leq \sigma^2$ and $\mathbb{E}[e_k\mid h_k] = 0$ for all $k\ge 1$, 
 implying that $\mathbb{E}[\|e_{k}\|^2] \leq \sigma^2$ for all $k\ge 1$, and
\[\mathbb{E}[\langle e_k,  h_k-u_k^*\rangle] =
\mathbb{E}\left[ \mathbb{E}[\langle e_k,  h_k-u_k^*\rangle\mid h_k,u_k^*] \right] =
0\qquad\hbox{for all }k\ge0.\]
Therefore, by taking the expectation in relation~\eqref{eq-th-linear-quasi-rate-002} we obtain for all $k\ge1$,
\begin{equation}
\begin{aligned}
\label{eq-th-linear-quasi-rate-003}
\mathbb{E}[\dist^2(u_{k+1},U^*)] \leq 
&\mathbb{E}[\dist^2(u_k,U^*)] - \mathbb{E}[\|u_k - h_k\|^2]  -2\alpha_k\mu \mathbb{E}[\dist^2(h_k,U^*)]  \\
&+ 24 \alpha_{k}^2 C^2(  \mathbb{E}[\|h_k\|^2] +  \mathbb{E}[\|h_{k-1}\|^2]) 
+ 6 \alpha_k^2 ( 2\sigma^2 + 4 D^2).
\end{aligned}
\end{equation}
Since the stepsize $\a_k=\frac{2}{\mu(2+k)}$ satisfies Assumption~\ref{asum-steps}, the conditions of Theorem~\ref{Theorem-Linear-sharp-ASC} are satisfied. Thus,
by Theorem ~\ref{Theorem-Linear-sharp-ASC}(b), the sequence $\{E \|h_k\|^2\}$ is bounded, so there exists a constant $M > 0$ such that $\mathbb{E}[\|h_k\|^2] \leq M$ for all $k \geq 1$. Thus, we have for all $k\ge1,$
\begin{equation}
\begin{aligned}
\label{eq-th-linear-quasi-rate-step1}
\mathbb{E}[\dist^2(u_{k+1}, U^*)]\leq &
\mathbb{E}[\dist^2(u_k,U^*)] 
- \mathbb{E}[\|u_k - h_k\|^2] -2\alpha_k\mu \mathbb{E}[\dist^2(h_k,U^*)] \cr
+&  12\alpha_{k}^2 ( 4 C^2 M + \sigma^2 + 2 D^2).
\end{aligned}
\end{equation}

Next, we estimate the term $\mathbb{E}[\dist^2(h_k,U^*)]$ in~\eqref{eq-th-linear-quasi-rate-step1}. By the triangle inequality we have
\[\|u_k - u^*\| \leq \|u_k - h_k\| + \|h_k - u^*\|\qquad\hbox{for all }u^*\in U^*,\]
and by taking the minimum over $u^*\in U^*$ on both sides of the preceding relation, we obtain
\begin{equation}\label{eq-estimate-dist0}
\dist(u_k,U^*)\le \|u_k - h_k\| +\dist(h_k,U^*).
\end{equation}
By using the inequality $(\sum_{i=1}^m a_i)^2\le m\sum_{i=1}^m a_i^2$, which is valid for any scalars $a_i$, $i=1,\ldots,m,$ and any integer $m\ge 1$,
we further obtain
\[\dist^2(u_k,U^*)\le (\|u_k - h_k\| +\dist(h_k,U^*))^2\le 2\|u_k - h_k\|^2 + 2\dist^2(h_k,U^*).\]
Hence, it follows that
\begin{equation}\label{eq-estimate-dist}
- 2 \dist^2(h_k, U^*) 
\leq 2 \|u_k - h_k\|^2 - \dist^2(u_k, U^*).
\end{equation}
By using the preceding estimate in relation~\eqref{eq-th-linear-quasi-rate-step1}, we can find that for all $k\ge 1$,
\begin{equation}
\begin{aligned}
\label{eq-th-linear-quasi-rate-4}
\mathbb{E}[\dist^2(u_{k+1}, U^*)]\leq & 
(1-\alpha_k\mu)\mathbb{E}[\dist^2(u_k,U^*)] 
- (1-2\alpha_k\mu)\mathbb{E}[\|u_k - h_k\|^2]  \cr
+& 12\alpha_{k}^2 ( 4 C^2 M + \sigma^2 + 2 D^2).
\end{aligned}
\end{equation}
The stepsize $\a_k=\frac{2}{\mu(3+k)}$, $k\ge1$, satisfies $a_k\le \frac{1}{2\mu}$ for all $k\ge 1$, implying that $2\alpha_k\mu\le 1$. Thus, we obtain that for all $k\ge 1$,
\begin{equation}
\label{eq-th-linear-quasi-rate-5}
\mathbb{E}[\dist^2(u_{k+1}, U^*)]\leq 
(1-\alpha_k\mu)\mathbb{E}[\dist^2(u_k,U^*)] 
+ 12\alpha_{k}^2 ( 4 C^2 M + \sigma^2 + 2 D^2).
\end{equation}
We next show that we can apply Lemma~\ref{Lemma7-stich} to the relation~\eqref{eq-th-linear-quasi-rate-5}. To do so, at first, we note that $\a_k\le \frac{1}{2\mu}\le \frac{2}{3\mu}$ for all $k\ge 1$. Thus, we have that $\a_k\le d^{-1}$ with $d=3\mu/2$ for all $k\ge1$.
We let $a=\mu$ and note that $\frac{2d}{a}=3$, so that the stepsize satisfies
\[\a_k=\frac{2}{\mu(3+k)}=\frac{2}{a\left(\frac{2d}{a} +k\right)}\qquad\hbox{for all } k\ge 1.\]
Thus, Lemma~\ref{Lemma7-stich} applies to relation~\eqref{eq-th-linear-quasi-rate-5} with a time shift to start with $k=1$ instead of $k=0$ and   with the following identification 
\[r_k=\mathbb{E}[\dist^2(u_k, U^*)],\quad
s_k=0,\quad \g_k=\a_k,\quad a=\mu,\quad d=\frac{3\mu}{2},\quad c=12( 4 C^2 M + \sigma^2 + 2 D^2).\]
Hence, by Lemma~\ref{Lemma7-stich} (starting with $k=1$ instead of $k=0$) we obtain that
for any \dv{$K\ge 2$},
\[\mu \mathbb{E}[\dist^2(u_{K+1}, U^*)]
\leq  \frac{8d^2}{a (K-1)^2}\,\mathbb{E}[\dist^2(u_1, U^*)]  + \frac{2c}{a (K-1)}.
\]
Upon dividing with $\mu$ and substituting the expressions for $d$ and $c$, we have that  
\[\mathbb{E}[\dist^2(u_{K+1}, U^*)]
\leq \frac{18}{(K-1)^2}\,\mathbb{E}[\dist^2(u_1, U^*)] + \frac{24}{\mu^2 (K-1)} ( 4 C^2 M + \sigma^2 + 2 D^2).\]
\end{proof}


\subsection{Proof of Theorem \ref{Theorem-Linear-quasi-rate-exp}}
\begin{proof}
By equation \eqref{eq-th-linear-asc-2} in the proof of Theorem ~\ref{Theorem-Linear-sharp-ASC}, the following relation holds for all $u^* \in U^*$ and for all $k \geq 1$,
\begin{equation}
\begin{aligned}
\label{eq-th-linear-quasi-rate-exp-00}
\|u_{k+1} - u^*\|^2 + \|u_{k+1} - h_k\|^2\leq & (1 + 144 \alpha_k^2 C^2 )\|u_k  - u^*\|^2 - 2 \alpha_k \mu \dist^2 (h_k, U^*)\\
-&  (1 - 72 \alpha_k^2 C^2) \|u_k - h_k\|^2 
+ 2 \alpha_k \langle  e_k, u^* - h_k\rangle \\
+ & 72 \alpha_{k}^2 C^2 \|u_k - h_{k-1}\|^2 \\
+ & 6 \alpha_{k}^2 (\|  e_{k-1}\|^2 + \| e_{k}\|^2   
+ 24 \|u^*\|^2 + 4  D^2) .
\end{aligned}
\end{equation}

The solution set $U^*$ is closed, so there exists a projection $u_k^*$ of the iterate $u_k$ on the optimal set $U^*$, i.e., there is a point $u_k^*\in U^*$ such that $\|u_k- u_k^*\|= \dist(u_k, U^*)$. Thus, by letting $u^*=u_k^*$ in relation~\eqref{eq-th-linear-quasi-rate-exp-00}, 
and noting that $\dist(u_{k+1},U^*)\le \|u_{k+1}-u_k^*\|$
we obtain for all $k\ge 1$,
\begin{equation}
\begin{aligned}
\label{eq-th-linear-quasi-rate-exp-01}
\dist^2(u_{k+1}, U^*) + \|u_{k+1} - h_k\|^2\leq & (1 + 144 \alpha_k^2 C^2 )\dist^2(u_k, U^*)- 2 \alpha_k \mu \dist^2 (h_k, U^*)\\
-&  (1 - 72 \alpha_k^2 C^2) \|u_k - h_k\|^2 
+ 2 \alpha_k \langle  e_k, u^*_k - h_k\rangle \cr
+  & 72 \alpha_{k}^2 C^2 \|u_k - h_{k-1}\|^2 \\
+ & 6 \alpha_{k}^2 (\|  e_{k-1}\|^2 + \| e_{k}\|^2   
+ 24 \|u_k^*\|^2 + 4  D^2) .
\end{aligned}
\end{equation}
We next estimate the term $-2\dist^2 (h_k, U^*)$ in ~\eqref{eq-th-linear-quasi-rate-exp-01} by using the relation shown in~\eqref{eq-estimate-dist}, i.e.,
\[- 2 \dist^2(h_k, U^*) \leq 2 \|u_k - h_k\|^2 - \dist^2(u_k, U^*).\]
By substituting the preceding estimate in relation~\eqref{eq-th-linear-quasi-rate-exp-01}, we obtain for all $k\ge1$,
\begin{equation}
\begin{aligned}
\label{eq-th-linear-quasi-rate-exp-02}
\dist^2(u_{k+1}, U^*) + \|u_{k+1} - h_k\|^2\leq & (1 + 144 \alpha_k^2 C^2  - \mu \alpha_k)\dist^2(u_k, U^*) \cr
+&  72 \alpha_{k}^2 C^2 \|u_k - h_{k-1}\|^2 \cr
-& (1 - 2 \mu \alpha_k- 72 \alpha_k^2 C^2 ) \|u_k - h_k\|^2 
+ 2 \alpha_k \langle  e_k, u_k^* - h_k\rangle \cr
+ & 6 \alpha_{k}^2 (\|  e_{k-1}\|^2 + \| e_{k}\|^2)   
+ 6 \alpha_{k}^2 (24 \|u_k^*\|^2 + 4  D^2).
\end{aligned}
\end{equation}

By Assumption ~\ref{asum-samples}, we have that $\mathbb{E}[\|e_{k}\|^2\mid h_k] \leq \sigma^2$ and $\mathbb{E}[e_k\mid h_k] = 0$ for all $k\ge 1$, 
 implying that $\mathbb{E}[\|e_{k}\|^2] \leq \sigma^2$ for all $k\ge 1$, and
\[\mathbb{E}[\langle e_k,  h_k-u_k^*\rangle] =
\mathbb{E}\left[ \mathbb{E}[\langle e_k,  h_k-u_k^*\rangle\mid h_k,u_k^*] \right] =
0\qquad\hbox{for all }k\ge1.\]
Therefore, by taking the expectation in relation~\eqref{eq-th-linear-quasi-rate-exp-02} and using the assumption that the set $U^*$ is bounded, we obtain for all $k\ge1$, 
\begin{equation}
\begin{aligned}
\label{eq-th-linear-quasi-rate-exp-2-0}
\mathbb{E}[\dist^2(u_{k+1}, U^*) + \|u_{k+1} - h_k\|^2 ]\leq & 
(1 + 144 \alpha_k^2 C^2 - \mu \alpha_k )\mathbb{E}[\dist^2(u_k, U^*)]  \cr 
+& 72\alpha^2_k C^2 \mathbb{E}[\|h_{k-1} - u_k\|^2]\\
- & (1 - 2\mu \alpha_k -72\alpha^2_k C^2) \mathbb{E}[\|u_k - h_k\|^2] \\
+ &12 \alpha_k^2 (\sigma^2 + 2 D^2 + 12 M_1^2),
\end{aligned}
\end{equation}
where $M_1>0$ is such that $\|u^*\|\le M_1$ for all $u^*\in U^*$.

By the stepsize choice we have that $\a_k\le d^{-1}$  with $d^{-1} \leq \frac{\mu}{288 C^2}$ for all $k \geq 0$, 
implying that $144 \alpha_k C^2\le \mu/2$, and consequently
\[1 + 144 \alpha_k^2 C^2 - \mu \alpha_k 
\leq 1 + \frac{\mu}{2}\alpha_k - \mu \alpha_k 
= 1- \frac{\mu}{2}\alpha_k\qquad\hbox{for all }k\ge0.\]
Thus, it follows that
\begin{equation}
\begin{aligned}
\label{eq-th-linear-quasi-rate-exp-2}
\mathbb{E}[\dist^2(u_{k+1}, U^*) + \|u_{k+1} - h_k\|^2 ]
\leq & 
\left(1 - \frac{\mu}{2}\alpha_k\right) \mathbb{E}[\dist^2(u_k, U^*)]  + 72\alpha^2_k C^2 \mathbb{E}[\|h_{k-1} - u_k\|^2]\\
- & (1 - 2\mu \alpha_k -72\alpha^2_k C^2) \mathbb{E}[\|u_k - h_k\|^2] \cr
+ & 12 \alpha_k^2 (\sigma^2 + 2 D^2 + 12 M_1^2).
\end{aligned}
\end{equation}
Since $\alpha_k \leq \frac{\mu}{288 C^2}$ for all $k \geq 0$, we also have 
\[72\a_k C^2\le\frac {\mu}{4}\quad\implies\quad
72\a_k^2 C^2\le\frac {\mu}{4}\a_k.\]
Therefore
\begin{equation}
    \label{eq-estim11}
1-2\mu\a_k-72\a_k^2 C^2
\ge 1-2\mu\a_k-\frac {\mu}{4}\a_k
=1-\frac{9}{4}\mu\a_k\ge0,\end{equation}
where the last inequality follows from the stepsize choice so that $\a_k\le d^{-1}$, and our assumption that $d^{-1}\le \frac{4}{9\mu}$ for all $k$.
Moreover, from $1-2\mu\a_k-72\a_k^2 C^2\ge0$, it follows that 
\begin{equation}
    \label{eq-estim12}
    72\a_k^2 C^2\le 1-2\mu\a_k<1-\frac{\mu}{2}\a_k\qquad\hbox{for all }k\ge0.\end{equation}
By using the estimates~\eqref{eq-estim11} and~\eqref{eq-estim12} in relation~\eqref{eq-th-linear-quasi-rate-exp-2} we obtain that for all $k\ge1$,
\begin{equation}
\begin{aligned}
\label{eq-th-linear-quasi-rate-exp-4}
\mathbb{E}[\dist^2(u_{k+1}, U^*) + \|u_{k+1} - h_k\|^2 ]\leq & 
\left(1 - \frac{\mu}{2}\alpha_k\right) 
\left(\mathbb{E}[\dist^2(u_k, U^*)]  + \mathbb{E}[\|h_{k-1} - u_k\|^2]\right)\\
& + 12 \alpha_k^2 (\sigma^2 + 2 D^2 + 12 M_1^2).
\end{aligned}
\end{equation}
The equation~\eqref{eq-th-linear-quasi-rate-exp-4} satisfies the conditions of Lemma~\ref{Lemm3-stich}  with the following identification 
\[r_k=\mathbb{E}[\dist^2(u_{k+1}, U^*) + \|u_{k+1} - h_k\|^2 ],\qquad
s_k=0,\qquad \g_k=\a_k,\qquad a=\frac{\mu}{2},\]
\[d\ge\frac{1}{\min\{ \frac{\mu}{288 C^2},\frac{4}{9\mu}\}},\qquad c=12 (\sigma^2 + 2 D^2 + 12 M_1^2).\]
By applying Lemma~\ref{Lemm3-stich} with a time shift to start with $k=1$ instead of $k=0$, we obtain that for all \dv{$K\ge2$},
\begin{align}
a \mathbb{E}[\dist^2(u_{K+1}, U^*) + \|u_{K+1} - h_K\|^2 ]\leq & 
32d (\mathbb{E}[\dist^2(u_1, U^*) + \|h_{0} - u_1\|^2]) e^{-\frac{a(K-1)}{2d}} \cr
+& \frac{36c}{a(K-1)}.
\end{align}
Upon dividing by $a=\frac{\mu}{2}$ and omitting the term $\|u_{k+1} - h_k\|^2 $, we arrive at
\[\mathbb{E}[\dist^2(u_{K+1}, U^*)]
\leq  
\frac{64d}{\mu} \left(\mathbb{E}[\dist^2(u_1, U^*) + \|h_{0} - u_1\|^2]\right) e^{-\frac{\mu(K-1)}{4d}} + \frac{144c}{\mu^2(K-1)}.\]
\end{proof}

\subsection{Proof of Theorem~\ref{Theorem-Lipschitz-quasi-rate-exp}}\label{sec-lipmap}
\begin{proof}
By Lemma \ref{Lemma1} we have that surely  for all $y \in U$ and $k \geq 1$, 
\begin{align}
\label{eq-lemma5-0}
\|u_{k+1} - y\|^2  \leq & \|u_k  - y\|^2  - \|u_{k+1} - h_k\|^2 -  \|u_k - h_k\|^2  - 2 \alpha_k \langle  F(h_{k}), h_k-y\rangle \cr
&-2 \alpha_k \langle  e_k, h_k-y\rangle
+ 6 \alpha_k^2 \|F(h_{k-1}) -F(h_{k})\|^2 + 6 \alpha_{k}^2 
\left(\|  e_{k-1}\|^2 + \|  e_{k}\|^2 \right).  
\end{align}
Using the Lipschitz continuity of the operator $F(\cdot)$ we can bound the term $ \|F(h_k) -  F(h_{k-1}) \|^2$, as follows
\begin{equation}
\begin{aligned}
\label{eq-lemma5-1}
 \|F(h_k) -  F(h_{k-1}) \|^2  
 &\leq  L^2\|h_{k-1} - h_k\|^2 \\
 &\le L^2\left(\|h_{k-1} - u_k\| +\|u_k- h_k\|\right)^2\cr
&\leq  2 L^2 (\|h_{k-1} -u_{k}\|^2 + \|u_k - h_k\|^2),
\end{aligned}
\end{equation}
where the last inequality follows the inequality $(\sum_{i=1}^m a_i)^2\le m\sum_{i=1}^m a_i^2$, which is valid for any scalars $a_i$, $i=1,\ldots,m,$ and any integer $m\ge 1$. Combining 
relations ~\eqref{eq-lemma5-0} and~\eqref{eq-lemma5-1}, and letting $y=u^* \in U^*$, we surely obtain
for all $u^* \in U^*$ and $k \geq 1$, 
\begin{equation}
\begin{aligned}
\label{eq-lemma5-2-0}
\|u_{k+1} - u^*\|^2 + \|u_{k+1} - h_k\|^2 \leq & \|u_k  - u^*\|^2  -  (1 - 12 \alpha_k^2 L^2) \|u_k - h_k\|^2 \cr  
-&2 \alpha_k \langle  F(h_{k}), h_k -u^*\rangle -2 \alpha_k \langle  e_k, h_k - u^*\rangle \cr
+ & 12 \alpha_{k}^2 L^2 \|h_{k-1} -u_{k}\|^2 + 6 \alpha_{k}^2 \left(\|  e_{k-1}\|^2 + \|  e_{k}\|^2 \right).  
\end{aligned}
\end{equation}
By the $2$-quasi sharpness  property of $F(\cdot)$ (Assumption~\ref{asum-sharp}, with $p=2$), we have that 
$\langle   F(h_k), h_k - u^*  \rangle \ge  \dist^2(h_k, U^*),$
thus implying that 
\begin{equation}
\begin{aligned}
\label{eq-lemma5-2}
\|u_{k+1} - u^*\|^2 + \|u_{k+1} - h_k\|^2 \leq & \|u_k  - u^*\|^2 - 2 \alpha_k \mu \dist^2(h_k, U^*) \cr
-& (1 - 12 \alpha_k^2 L^2)\|u_k - h_k\|^2   -2 \alpha_k  \langle  e_k, h_k - u^*\rangle \cr
+& 12 \alpha_{k}^2 L^2 \|h_{k-1} -u_{k}\|^2 + 6 \alpha_{k}^2 \left(\|  e_{k-1}\|^2 + \|   e_{k}\|^2\right) .  
\end{aligned}
\end{equation}
Using the relation (see~\eqref{eq-estimate-dist})
\[- 2 \dist^2(h_k, U^*) 
\leq 2 \|u_k - h_k\|^2 - \dist^2(u_k, U^*),\]
we obtain surely for all $u^* \in U^*$ and $k \geq 1$, 
\begin{equation}
\begin{aligned}
\label{eq-lemma5-21}
\|u_{k+1} - u^*\|^2 + \|u_{k+1} - h_k\|^2 
\leq & (1-\mu\a_k)\|u_k  - u^*\|^2 - (1 -2\mu\a_k- 12 \alpha_k^2 L^2)\|u_k - h_k\|^2   \\
&-2 \alpha_k  \langle  e_k, h_k - u^*\rangle 
+ 12 \alpha_{k}^2 L^2 \|h_{k-1} -u_{k}\|^2 \cr
&+ 6 \alpha_{k}^2 \left(\|  e_{k-1}\|^2 + \|   e_{k}\|^2\right) .  
\end{aligned}
\end{equation}

Since $U^*$ is closed, there is a projection of $u^*_k$ of the iterate $u_k$ on the solution set $U^*$ such that $\|u_k- u_k^*\| = \dist (u_k, U^*)$. Thus, by letting $u^* = u^*_k$ and by noting that 
$\dist(u_{k+1},U^*)\le \|u_{k+1}-u^*_k\|$, we can see that for all $k \geq 1$,
\begin{equation}
\begin{aligned}
\label{eq-th-lip-quasi-rate-2}
\dist^2(u_{k+1},U^*) + \|u_{k+1} - h_k\|^2 | 
\leq & 
(1-\mu\a_k)\dist^2(u_k,U^*) \\
& - (1 -2\mu\a_k- 12 \alpha_k^2 L^2)\|u_k - h_k\|^2  \\
& - 2 \alpha_k  \langle  e_k, h_k - u^*_k\rangle 
+ 12 \alpha_{k}^2 L^2 \|h_{k-1} -u_{k}\|^2  \cr
& + 6 \alpha_{k}^2 \left(\|  e_{k-1}\|^2 + \|   e_{k}\|^2\right).
\end{aligned}
\end{equation}
We next consider the coefficient $1 -2\mu\a_k- 12 \alpha_k^2 L^2$. We note that the polynomial 
$p(s)=1 -2\mu s- 12 L^2 s^2$, $s\in\mathbb{R}$, has two real roots $s_1{/2}=\frac{\mu\pm \sqrt{\mu^2+ 12 L^2}}{12 L^2}$. Thus, since the stepsize is selecetd so that $0<\a_k\le \frac{1}{2\sqrt 3 L}$ for all $k$  and since we have
\[\frac{1}{2\sqrt 3 L} =\frac{\sqrt{12 L^2}}{12 L^2}\le \frac{\mu+\sqrt{\mu^2+ 12 L^2}}{12 L^2},\] it follows that the stepsize $\a_k$ satisfies
\[1 -2\mu\a_k- 12 \alpha_k^2 L^2\ge0\qquad\hbox{for all } k\ge 0.\]
Subsequently, we have that $12 \alpha_k^2 L^2\le 1-2\mu\a_k<1-\mu \a_k$ for all $k$.
Thus, from~\eqref{eq-th-lip-quasi-rate-2} we obtain
surely for all $k\ge 1$,
\begin{equation}
\begin{aligned}
\label{eq-th-lip-quasi-rate-4}
\dist^2(u_{k+1}, U^*) + \|u_{k+1} - h_k\|^2 | \leq & 
(1 - \mu \alpha_k) \left(\dist^2(u_k, U^*) + \|h_{k-1} -u_{k}\|^2\right)\\
&-2 \alpha_k  \langle  e_k, h_k - u_k^*\rangle 
 + 6 \alpha_{k}^2 \left(\|  e_{k-1}\|^2 + \|   e_{k}\|^2\right).
\end{aligned}
\end{equation}

By Assumption~\ref{asum-samples}, we have that $\mathbb{E}[\|e_{k}\|^2\mid h_k] \leq \sigma^2$ and $\mathbb{E}[e_k\mid h_k] = 0$ for all $k\ge 1$. Therefore, 
$\mathbb{E}[\|e_{k}\|^2] \leq \sigma^2$ for all $k\ge 1$, and
\[\mathbb{E}[\langle e_k,  h_k-u_k^*\rangle] =
\mathbb{E}\left[ \mathbb{E}[\langle e_k,  h_k-u_k^*\rangle\mid h_k,u_k^*] \right] =
0\qquad\hbox{for all }k\ge1.\]
Hence, by taking the expectation in relation~\eqref{eq-th-lip-quasi-rate-4} we obtain for all $k\ge1$,
\begin{equation}
\begin{aligned}
\label{eq-th-lip-quasi-rate-5}
\mathbb{E}\left[\dist^2(u_{k+1}, U^*) + \|u_{k+1} - h_k\|^2 \right]
\leq & 
(1 - \mu \alpha_k )\mathbb{E}\left[\dist^2(u_k,U^*) + \|h_{k-1} - u_k\|^2\right] \cr
+& 12 \alpha_k^2 \sigma^2.
\end{aligned}
\end{equation}

Relation~\eqref{eq-th-lip-quasi-rate-5} satisfies the conditions of Lemma~\ref{Lemm3-stich}  with the following identification 
\[r_{k}=\mathbb{E}[\dist^2(u_{k}, U^*) + \|u_{k} - h_{k-1}\|^2 ],\qquad
s_k=0,\qquad \g_k=\a_k,\qquad a=\mu,\]
\[d\ge \max\left\{2\sqrt{3} L,\, \mu\right\},\qquad c=12 \sigma^2.\]
Thus, by using Lemma~\ref{Lemm3-stich} with a time shift to start with $k=1$ instead of $k=0$,
we obtain for all \dv{$K\ge2$},
\begin{equation}
\begin{aligned}
\mu \mathbb{E}[\dist^2(u_{K+1}, U^*) + \|u_{K+1} - h_K\|^2 ]\leq & 
32d (\mathbb{E}[\dist^2(u_1, U^*) + \|h_{0} - u_1\|^2]) e^{-\frac{\mu(K-1)}{2d}} \cr
+& \frac{36c}{\mu(K-1)}.
\end{aligned}
\end{equation}
Upon dividing by $\mu$ and substituting $c=12 \sigma^2$, we find that for all $K\ge2$,
\begin{equation}
\begin{aligned}
\mathbb{E}\left[\dist^2(u_{K+1}, U^*) + \|u_{K+1} - h_K\|^2 \right]\leq&
\frac{32 d}{\mu} \left(\mathbb{E}[\dist^2(u_1, U^*) + \|h_{0} - u_1\|^2]\right) e^{-\frac{\mu(K-1)}{2d}} \cr
+& \frac{432 \sigma^2 }{\mu^2(K-1)},
\end{aligned}
\end{equation}

which implies the stated relation.
\end{proof}


\subsection{Proof of Theorem \ref{Theorem-Linear-sharp-rate}}\label{sec-proofThm-1sharp}

\begin{proof}
By Lemma~\ref{Lemma1} we surely have for all $y\in U$ and all $k\ge1$,
\begin{align*}
\|u_{k+1}  - y\|^2 
\leq &\|u_k  - y\|^2 - \|u_{k+1} -h_k\|^2 - \|u_k - h_k\|^2  - 2 \alpha_{k} \langle e_k+ F(h_k), h_k - y\rangle \\
&+ 6 \alpha_{k}^2 \,\|F(h_k) -  F(h_{k-1}) \|^2
+ 6\alpha_{k}^2 ( \|e_k\| + \|e_{k-1}\|^2).
\end{align*}
Since the set $U$ is compact and the operator $F(\cdot)$ is continuous, it follows by Corollary 2.2.5 in~\cite{facchinei2003finite} that the solution set $U^*$ of the SVI$(U,F)$ is a nonempty and compact set. Therefore, by letting $u=u^*$ with $u^*\in U^*$ in the preceding relation, we surely obtain for all $u^*\in U^*$ and all $k\ge1$,
\begin{align*}
\|u_{k+1}  - u^*\|^2 
\leq &
\|u_k  - u^*\|^2 - \|u_{k+1} -h_k\|^2 - \|u_k - h_k\|^2  - 2 \alpha_{k} \langle e_k+ F(h_k), h_k - u^*\rangle \\
&+ 6 \alpha_{k}^2 \,\|F(h_k) -  F(h_{k-1}) \|^2
+ 6\alpha_{k}^2 ( \|e_k\| + \|e_{k-1}\|^2).
\end{align*}
By assumption there exist a constant $D>0$ such that $\|F(u)\| \leq D$ for all $u \in U$. Moreover, we have that $\{h_k\}\subset U$, implying that surely 
for all $k\ge 1$,
\[\|F(h_k)-F(h_{k-1})\|^2\le \left(\|F(h_k)\| +\|F(h_{k-1})\|\right)^2\le 4D^2.\]
By combining  the preceding two relations, and using the $p$-quasi sharpness property of the operator, we obtain that surely for all $u^*\in U^*$ and all $k\ge1$,
\begin{align}
\label{eq-estimate-thm38-1}
\|u_{k+1}  - u^*\|^2 
\leq & \|u_k  - u^*\|^2 - \|u_{k+1} -h_k\|^2 - \|u_k - h_k\|^2  -2\a_k\mu\,\dist^{\textcolor{black}{p}}(h_k,U^*) \cr
-& 2 \alpha_{k} \langle e_k, h_k - u^*\rangle 
+ 24 \alpha_{k}^2 \, D^2
+ 6\alpha_{k}^2 ( \|e_k\| + \|e_{k-1}\|^2).
\end{align}

Next, we estimate $\dist(h_k,U^*)$.
Since $U$ is a compact set, there is an $M_U>0$ such that 
\[\|u-u'\|\le M_U\qquad\hbox{for all }u,u'\in U.\]
Therefore, for any $u^*\in U^*\subseteq U$, we have
\[\dist(h_k,U^*)\le \|h_k-u^*\|\le M_U,\]
which implies that 
\begin{equation}
\label{eq-estimate-dist01}
\dist^2(h_k,U^*)\le M_U^{\textcolor{black}{2-p}}\dist^{\textcolor{black}{p}}(h_k,U^*)
\qquad\implies\qquad 
- \dist^{\textcolor{black}{p}}(h_k,U^*)\le -\frac{1}{M_U^{\textcolor{black}{2-p}}}\dist^2(h_k,U^*).
\end{equation}
Moreover, by using the following relation (see~\eqref{eq-estimate-dist})
\begin{equation*}
-2\dist^2(h_k,U^*)\le 2\|u_k - h_k\|^2 -\dist^2(u_k,U^*),
\end{equation*}
from~\eqref{eq-estimate-dist01} we obtain that
\begin{equation}
\label{eq-estimate-dist02}
- 2\dist(h_k,U^*)\le \frac{1}{M_U^{\textcolor{black}{2 -p}}}\left( 2\|u_k - h_k\|^2 -\dist^2(u_k,U^*)\right).
\end{equation}
Bu using~\eqref{eq-estimate-dist02} in relation~\eqref{eq-estimate-thm38-1}
we surely obtain for all $u^*\in U^*$ and $k\ge1$,
\begin{align}
\label{eq-estim-222}
\|u_{k+1}  - u^*\|^2 
\leq & \|u_k  - u^*\|^2 - \|u_{k+1} -h_k\|^2 - \|u_k - h_k\|^2 \cr
& + \frac{\a_k\mu}{M_U^{\textcolor{black}{2-p}}}
\left( 2\|u_k - h_k\|^2 -\dist^2(u_k,U^*)\right)\cr
& - 2 \alpha_{k} \langle e_k, h_k - u^*\rangle + 24 \alpha_{k}^2 \, D^2
+ 6\alpha_{k}^2 ( \|e_k\| + \|e_{k-1}\|^2).
\end{align}
Next, we let $u^*=u^*_k$, where $u^*_k$ is a projection of $u_k$ on the solution set $U^*$, which exists since $U^*$ is a closed set. We also use $\|u_k- u_k^*\| = \dist (u_k, U^*)$ and $\dist(u_{k+1}, U^*) \leq\|u_{k+1} - u_{k}^*\|$ and, thus, after re-arranging the terms in~\eqref{eq-estim-222}
we obtain 
that surely for all $k\ge 1$,
\begin{align}
\label{eq-Theorem-Linear-sharp-rate-3}
\dist^2(u_{k+1},U^*) + \|u_{k+1} -h_k\|^2
\leq & \left(1-\frac{\a_k\mu}{M_U^{\textcolor{black}{2-p}}}\right)\dist^2(u_k,U^*)  -\left(1-\frac{2\a_k\mu}{ M_U^{\textcolor{black}{2 - p}}}\right)
\|u_k - h_k\|^2 \cr
& -2 \alpha_{k} \langle e_k, h_k - u^*\rangle + 24 \alpha_{k}^2 \, D^2
+ 6\alpha_{k}^2 ( \|e_k\| + \|e_{k-1}\|^2).
\end{align}
By Assumption ~\ref{asum-samples}, we have that $\mathbb{E}[\|e_{k}\|^2\mid h_k] \leq \sigma^2$ and $\mathbb{E}[e_k\mid h_k] = 0$ for all $k\ge 1$, 
 implying that $\mathbb{E}[\|e_{k}\|^2] \leq \sigma^2$ for all $k\ge 1$,and
\[\mathbb{E}[\langle e_k,  h_k-u_k^*\rangle] =
\mathbb{E}\left[ \mathbb{E}[\langle e_k,  h_k-u_k^*\rangle\mid h_k,u_k^*] \right] =
0\qquad\hbox{for all }k\ge1.\]
Therefore, by taking the expectation in relation~\eqref{eq-Theorem-Linear-sharp-rate-3} (and omitting the term $\|u_{k+1} - h_k\|^2$), we obtain for all $k\ge1$,
\begin{equation}
\begin{aligned}
\label{eq-Theorem-Linear-sharp-rate-4-0}
\mathbb{E}[\dist^2(u_{k+1}, U^*)] \leq & 
\left(1 -  \frac{\alpha_k \mu}{M_U^{\textcolor{black}{2- p}}}\right) \mathbb{E}[\dist^2(u_k, U^*)]  - \left(1 - \frac{2 \alpha_k  \mu}{M_U^{\textcolor{black}{2 - p}}} \right) \mathbb{E}[\|u_k - h_k\|^2] \\
& +12 \alpha_k^2 (\sigma^2 + 2 D^2). 
\end{aligned}
\end{equation}
We now consider the two stepsize choices in parts (a) and (b) separately.\\

\noindent(a) 
Since the stepsize is given by $\alpha_k = \frac{2 M_U^{\textcolor{black}{2 - p}}}{\mu(3 + k)}$ for all $k\ge0$,
it follows that $\alpha_k \le \frac{ M_U^{\textcolor{black}{2 - p}}}{2\mu}$ for all $k\ge 1$. Hence, $1-2\mu\a_k/M_U^{\textcolor{black}{2 - p}}\ge0$, implying that
\begin{equation}
\begin{aligned}
\label{eq-Theorem-Linear-sharp-rate-4-1}
\mathbb{E}[\dist^2(u_{k+1}, U^*)] \leq & 
\left(1 -  \frac{\alpha_k \mu}{M_U^{\textcolor{black}{2 - p}}}\right) \mathbb{E}[\dist^2(u_k, U^*)]  + 12 \alpha_k^2 (\sigma^2 + 2 D^2). 
\end{aligned}
\end{equation}
The equation ~\eqref{eq-Theorem-Linear-sharp-rate-4-1} satisfies the conditions of Lemma~\ref{Lemma7-stich}, where the recursive relation is starting with $k=1$ and with the following identification 
\begin{equation}
 \label{eq-identify1}   
r_{k}=\mathbb{E}[\dist^2(u_{k}, U^*)],\quad
s_k=0,\quad \g_k=\a_k,\quad a=\frac{\mu}{M_U^{\textcolor{black}{2 - p}}},\quad
d=\frac{2 \mu}{M_U^{\textcolor{black}{2 - p}}},\quad c=12 (\sigma^2 + 2 D^2 ) .
\end{equation}
By applying Lemma~\ref{Lemma7-stich}, with a time shift to start with $k=1$ instead of $k=0$, we find that for all \dv{ $K\ge 2$ }, 
\[a \mathbb{E}[\dist^2(u_{K+1}, U^*)]\leq  
\frac{8d^2}{ a(K-1)^2} (\mathbb{E}[\dist^2(u_1, U^*)])+ \frac{2c}{a(K-1)}.\]
Upon dividing by $a = \frac{\mu}{M_2}$, 
and substituting  the corresponding values for $d$ and $c$, we obtain
\[\mathbb{E}[\dist^2(u_{K+1}, U^*) ]\leq  
\frac{32 }{(K-1)^2} \mathbb{E}[\dist^2(u_1, U^*)]+ \frac{24(\sigma^2 + 2 D^2)M^2_U}{\mu^2(K-1)}.\]

\noindent(b) The stepsizes given by relations in~\eqref{eq-stich-step}, with  $a=\frac{\mu}{ M_U^{\textcolor{black}{2 - p}}}$ and $d =\frac{2 \mu}{M_U^{\textcolor{black}{2 - p}}}$, satisfies
$\a_k\le d^{-1}$ for all $k=0,1,\ldots,K-1$,
for any $K\ge1$.
Hence, we have 
$1-2\mu\a_k/M_U^{\textcolor{black}{2 - p}} \ge0$, implying that
relation~\eqref{eq-Theorem-Linear-sharp-rate-4-1} is valid for all $k\ge1$.
Thus, Lemma \ref{Lemm3-stich} applies with the same identification as in~\eqref{eq-identify1}.
Thus, by applying Lemma~\ref{Lemm3-stich}, with a time shift to start with $k=1$ instead of $k=0$, we obtain the following result for all $K\ge2$,
\[a \mathbb{E}[\dist^2(u_{K+1}, U^*)]\leq  
32d \mathbb{E}[\dist^2(u_1, U^*)] e^{ \frac{a(K-1)}{2d}} + \frac{36c}{a(K-1)},
\]
and after dividing by $a = \frac{\mu}{M_U^{\textcolor{black}{2 - p}}}$, we obtain
\[\mathbb{E}[\dist^2(u_{K+1}, U^*) ]\leq  
64\mathbb{E}[\dist^2(u_1, U^*)] e^{-\frac{(K-1)}{4}} + \frac{432 (\sigma^2 + 2 D^2) M_U^{\textcolor{black}{2(2-p)}}}{\mu^2(K-1)}.
\]
\end{proof}

\section{Experiment Details}
\label{Numeric-Details}
Firstly, we verify that the operator $F(\cdot)$ defined in~\eqref{eq-def-f} has linear growth, so it satisfies Assumption~\ref{asum-Linear}.
For this, we 
define the matrix $J$, as follows 
\[J = \begin{bmatrix}
A_1 & A_2\\
-A_2' & A_3
\end{bmatrix}.\] Then, for all $u \in U$ we have that
\begin{align*} \|F(u)\| 
&= \| c J u - c b\| \cr
& \leq c \|J u\| + c\|b\|
=c \sqrt{\langle J'J u, u \rangle} + c \|b\|, 
\end{align*}
where $c\le 1$ (see~\eqref{eq-def-f})  and
\begin{align*} 
J'J &= \begin{bmatrix} A_1' & - A_2 \\
A_2' & A_3'
\end{bmatrix} \begin{bmatrix} A_1 & A_2 \\
-A_2' &  A_3
\end{bmatrix}  = \begin{bmatrix} A_1'A_1 + A_2 A_2' & A_1' A_2 - A_2 A_3 \\
A_2' A_1 - A_3' A_2' &  A_2' A_2 + A_3' A_3
\end{bmatrix}. 
\end{align*}
Therefore,
the linear growth constant $C$ for the operator $F(\cdot)$ is the square root of the largest eigenvalue of the matrix $J'J$, i.e., $C = \sqrt{\lambda_{\max} (J'J)}$, and $D = \|b\|$.

Next, we show that operator is discontinuous when $\|u\| = 1$. Consider $v_k = u + \frac{1}{k} \mathbf{1}$ notice that as $k \rightarrow \infty$, $v_k \rightarrow u$, but 
\[\lim_{k \rightarrow \infty } \|F(u) - F(v_k)\| = \lim_{k \rightarrow \infty } \| J u - 0.5 J v_k + 0.5 b )\| \not = 0.\]
Hence, $F$ is discontinuous at $u$, when $\|u\| = 1$.

Now, we show that the operator $F(\cdot)$ 
satisfies Assumption~\ref{asum-sharp} with $p=2$. 
To show that $F(\cdot)$ is $2$-quasi sharp, we let $u \in\mathbb{R}^{\bd+\bs}$ be arbitrary, and $u^* \in U^*$ be an arbitrary solution, so $F(u^*) = 0$, and $J u^* + b = 0$. 
Let $F(u) = c J u + b$, note that we have 
\begin{align*} \langle F(u) , u - v \rangle 
&=  \langle F(u) - F(u^*) , u - u^* \rangle\cr
&= \langle c J(u  - u^*), u - u^*\rangle 
\end{align*}
Let $z = u - u^* \in \mathbb{R}^{\bd+\bs}$, and note that $z = [z_1, z_2]'$ with $z_1\in \mathbb{R}^{\bd}$ and $z_2\in\mathbb{R}^{\bs}$.
Thus, we have
\begin{align*} \langle z, J' z\rangle &= [z_1, z_2] \begin{bmatrix} A' & - B \\
B' & C'
\end{bmatrix} \begin{bmatrix}
z_1 \\
z_2
\end{bmatrix} \\
&= \langle z_1,  A_1 z_1 \rangle + \langle z_2, A_2' z_1 \rangle - \langle z_1,  A_2 z_2 \rangle + \langle z_2, A_3 z_2 \rangle \\
&= \langle z_1, A_1 z_1 \rangle + \langle z_2, A_3 z_2 \rangle \\
&\ge \mu_{A_1} \langle z_1, z_1 \rangle + \mu_{A_3} \langle z_2, z_2 \rangle \\
&\geq \min \{\mu_{A_1}, \mu_{A_3} \} \|z\|^2,
\end{align*}
where the first inequality in the preceding relation holds since $A_1$ and $A_3$ are symmetric positive definite matrices. 
Hence, we have that for arbitrary $u$ and arbitrary solution $u^* \in U^*$ it holds
\begin{align*} 
\la F(u), u - u^* \ra &\geq  c \min \{\mu_{A_1}, \mu_{A_3}\} \|u - u^*\|^2 \cr
&\geq 0.5 \min \{\mu_{A_1}, \mu_{A_3}\} \dist^2(u, U^*).
\end{align*}
Then, the operator $F(\cdot)$ is $2$-quasi sharp with  the constant $\mu = 0.5 \min \{ \mu_{A_1}, \mu_{A_3}\}$.

\subsection{Additional Experiments}
The results in Figure~\ref{fig:stichstepsizes} show that the Popov method reaches a small neighborhood of the solution with a smaller number of iterations than the prescribed threshold of $k_0 = 5000$ iterations according to the stepsize update rule in~\eqref{eq-lemma3stitchstep}. To get better convergence, we conducted another experiment, where we set the threshold $k_0$ to be equal to 200 iterations for both methods. The results of this experiment are shown in Figure~\ref{fig:trashold200}.

\begin{figure*}[ht]
\centering
\subfigure[$\kappa_F=5.1$]{
\includegraphics[width=.3\textwidth]{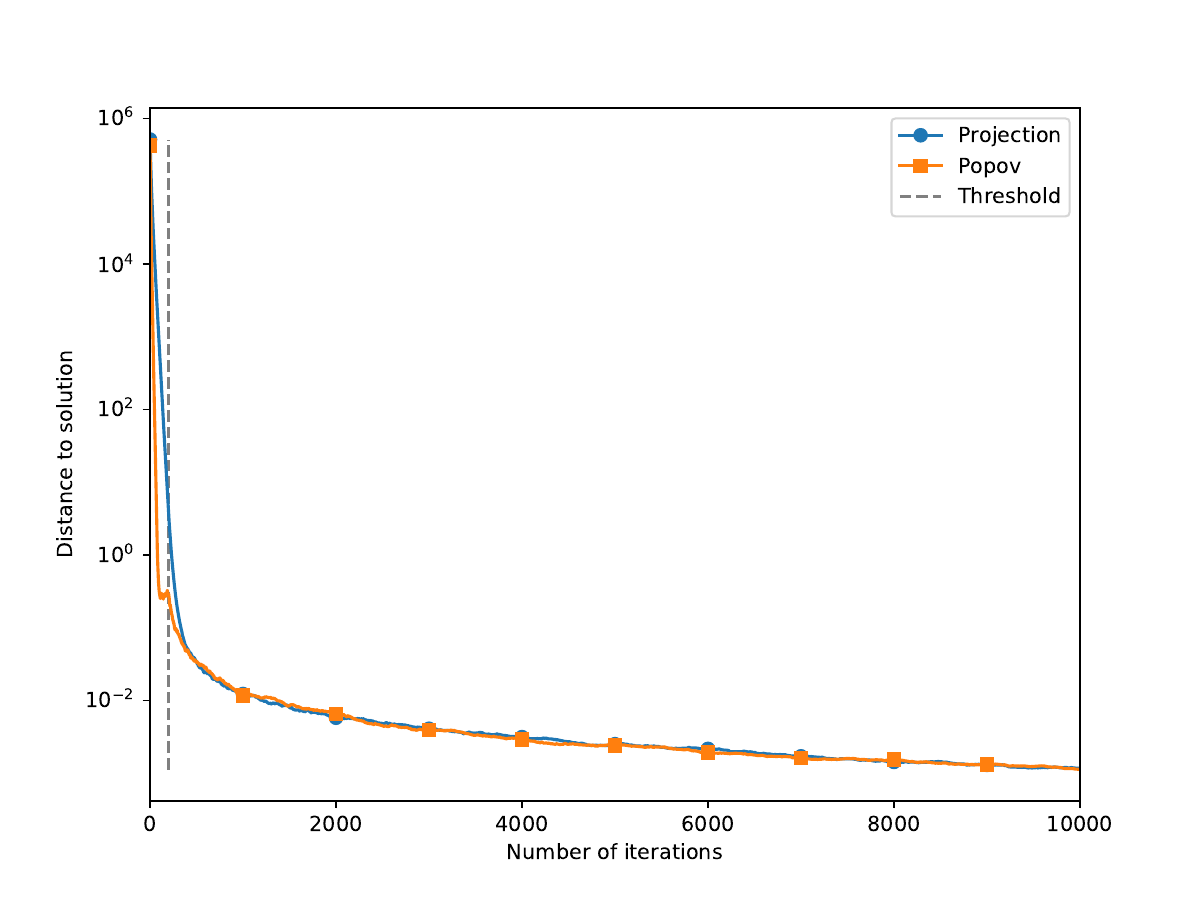}
}
\subfigure[$\kappa_F=51.5$]{
\includegraphics[width=.3\textwidth]{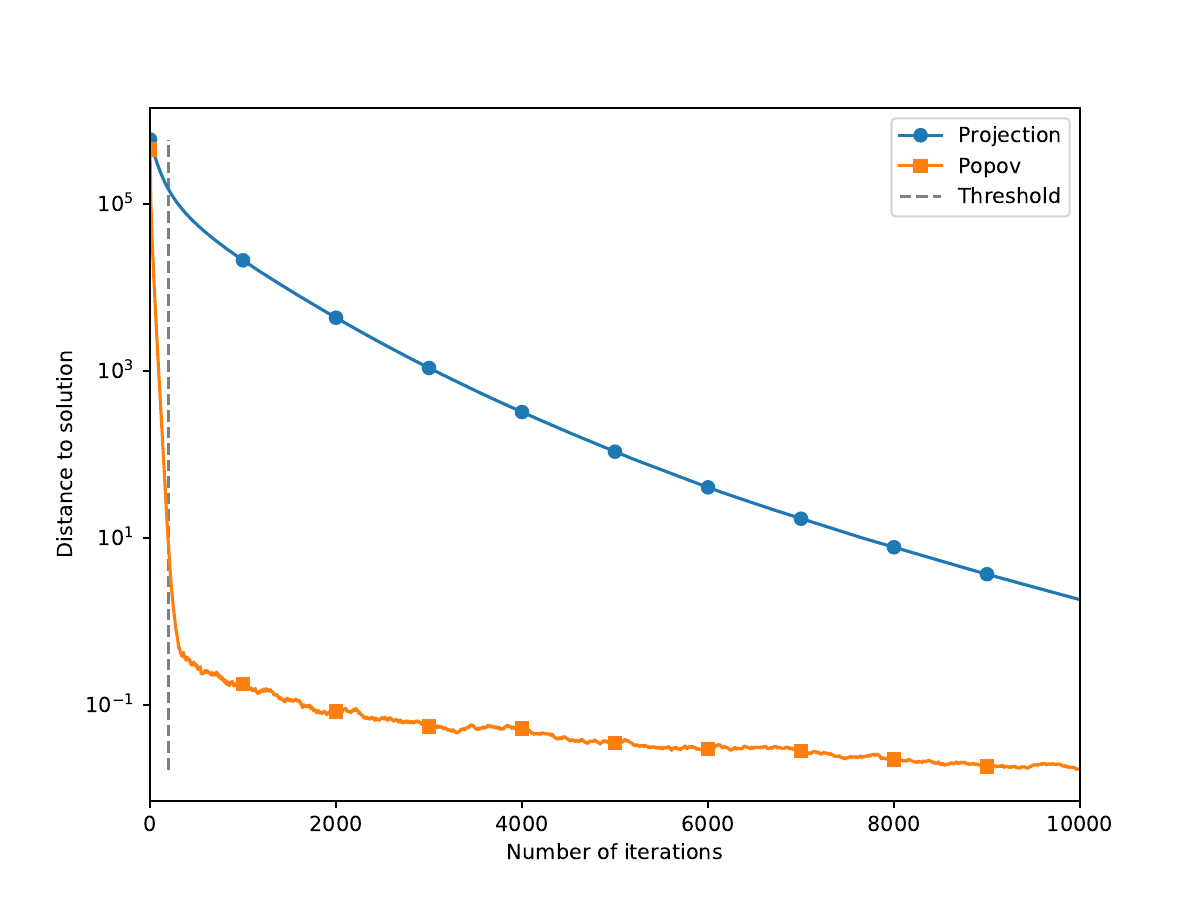}
}
\subfigure[$\kappa_F = 515.5$]{
\includegraphics[width=.3\textwidth]{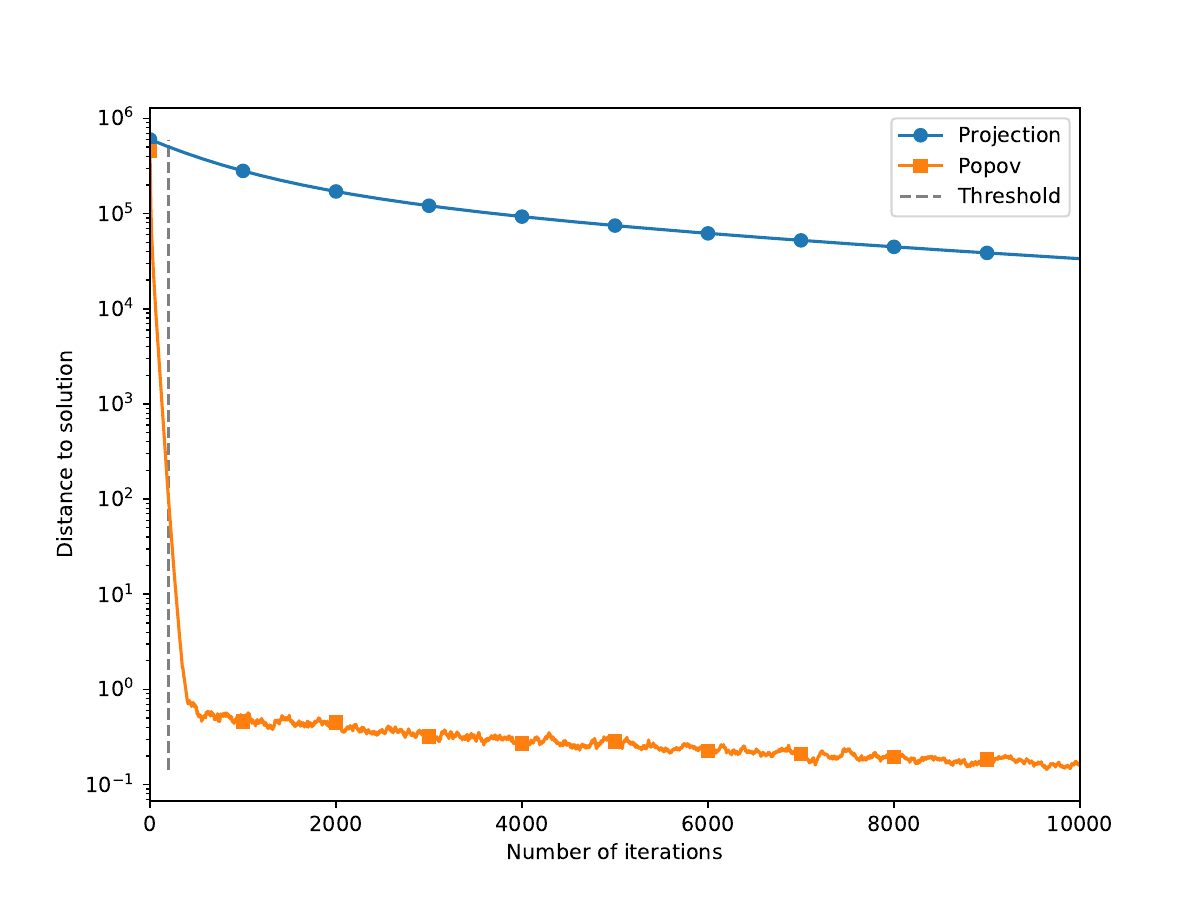}
}
\caption{Comparison of Popov method and projection method with stepsize rule given in \eqref{eq-lemma3stitchstep} and threshold $k_0=200$.}
\label{fig:trashold200}
\end{figure*}

Additionaly, we  consider the following  finite-sum min-max game with quadratic pay-off function  as in~\cite{loizou2021stochastic}:
\[\min_{u_1 \in \mathbb{R}^\bd} \max_{u_2 \in \mathbb{R}^\bs} \frac{1}{n} \sum_{i=1}^n f_i(u_1,u_2),\]
where for each $i=1,\ldots,n$, the function $f_i(\cdot)$ is given by
\begin{align*}
f_i(u_1,u_2)= &
    \langle u_1,  A_i u_1 \rangle + \langle u_1,  B_i u_2 \rangle - \langle u_2,  C_i u_2 \rangle +\langle \mathbf{a}_i,u_1 \rangle - \langle \mathbf{c}_i,  u_2 \rangle.
\end{align*}

\begin{figure*}[ht]
\centering
\subfigure[$\kappa_F=5.1$]{
\includegraphics[width=.3\textwidth]{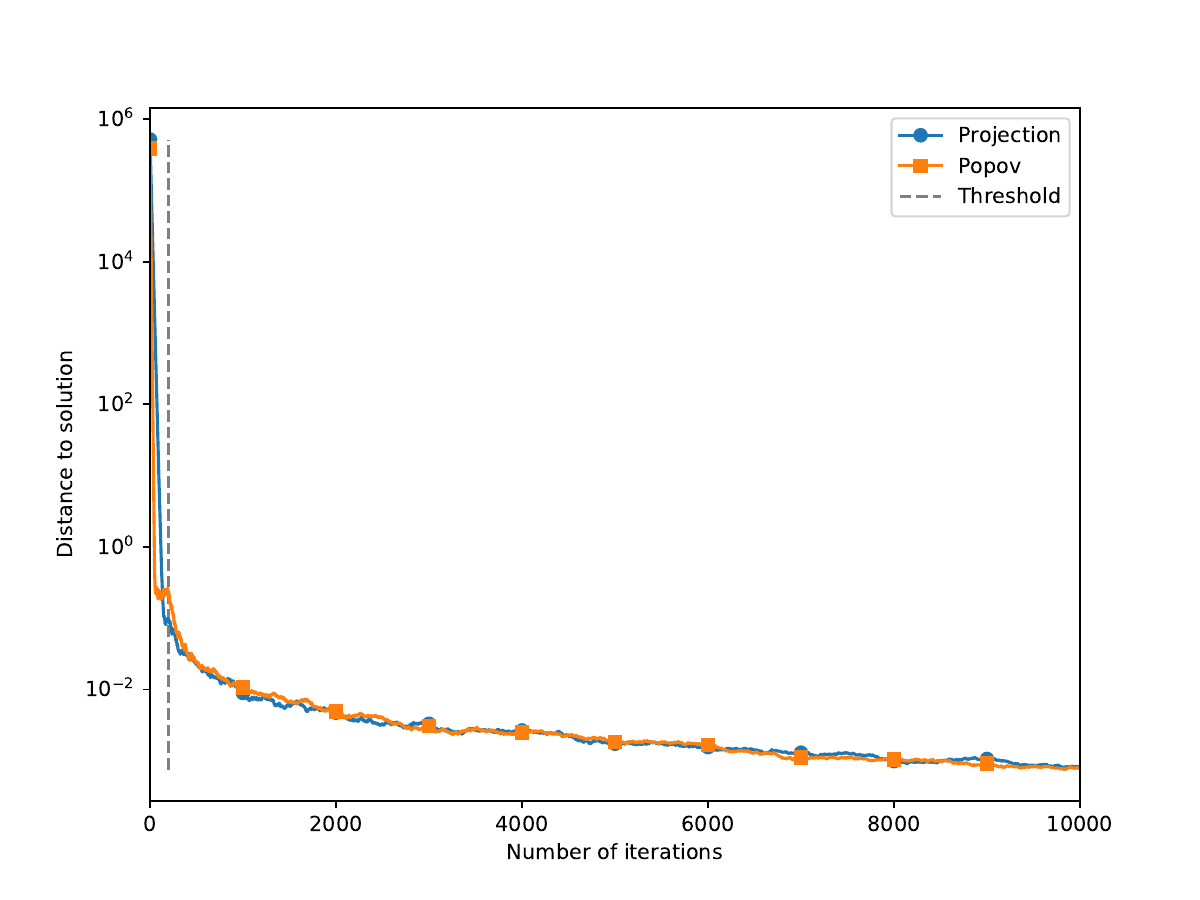}
}
\subfigure[$\kappa_F=50.6$]{
\includegraphics[width=.3\textwidth]{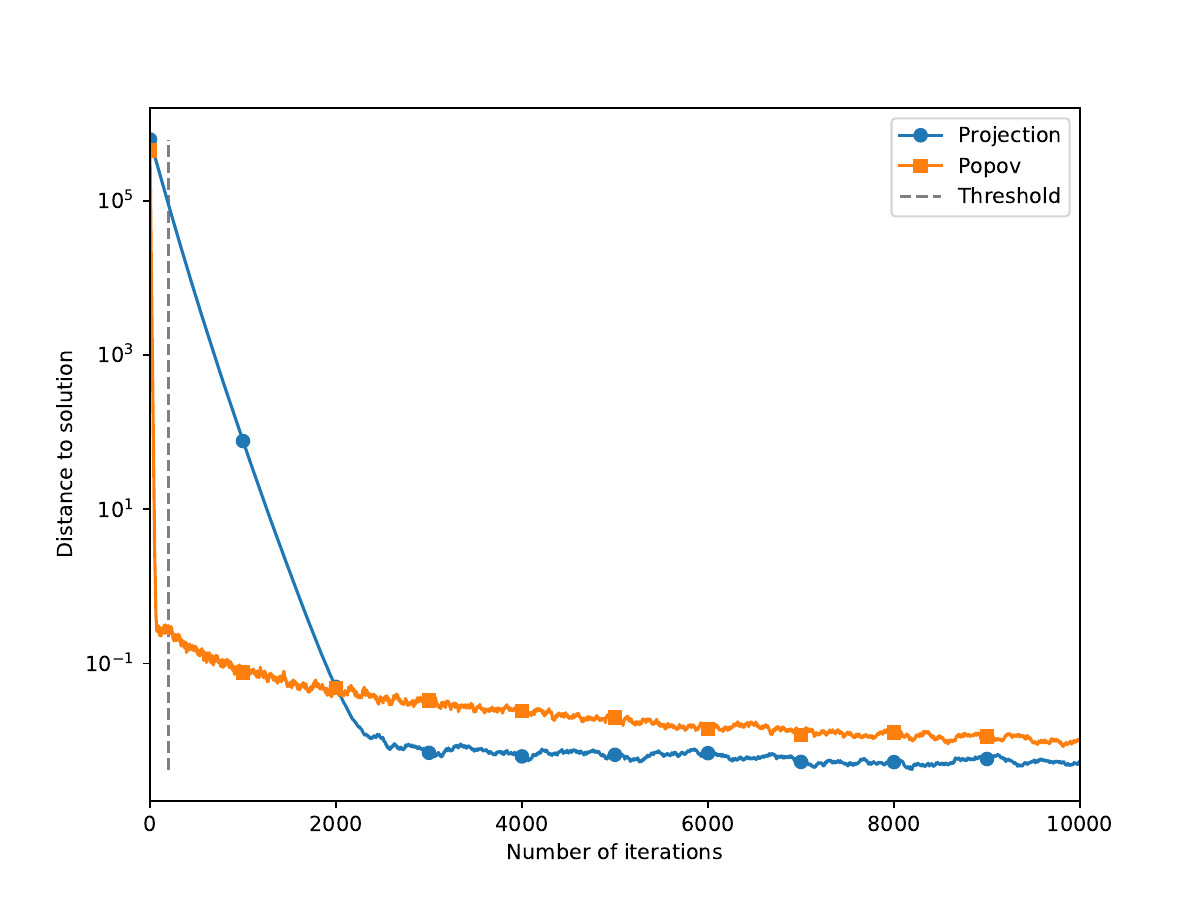}
}
\subfigure[$\kappa_F = 506.4$]{
\includegraphics[width=.3\textwidth]{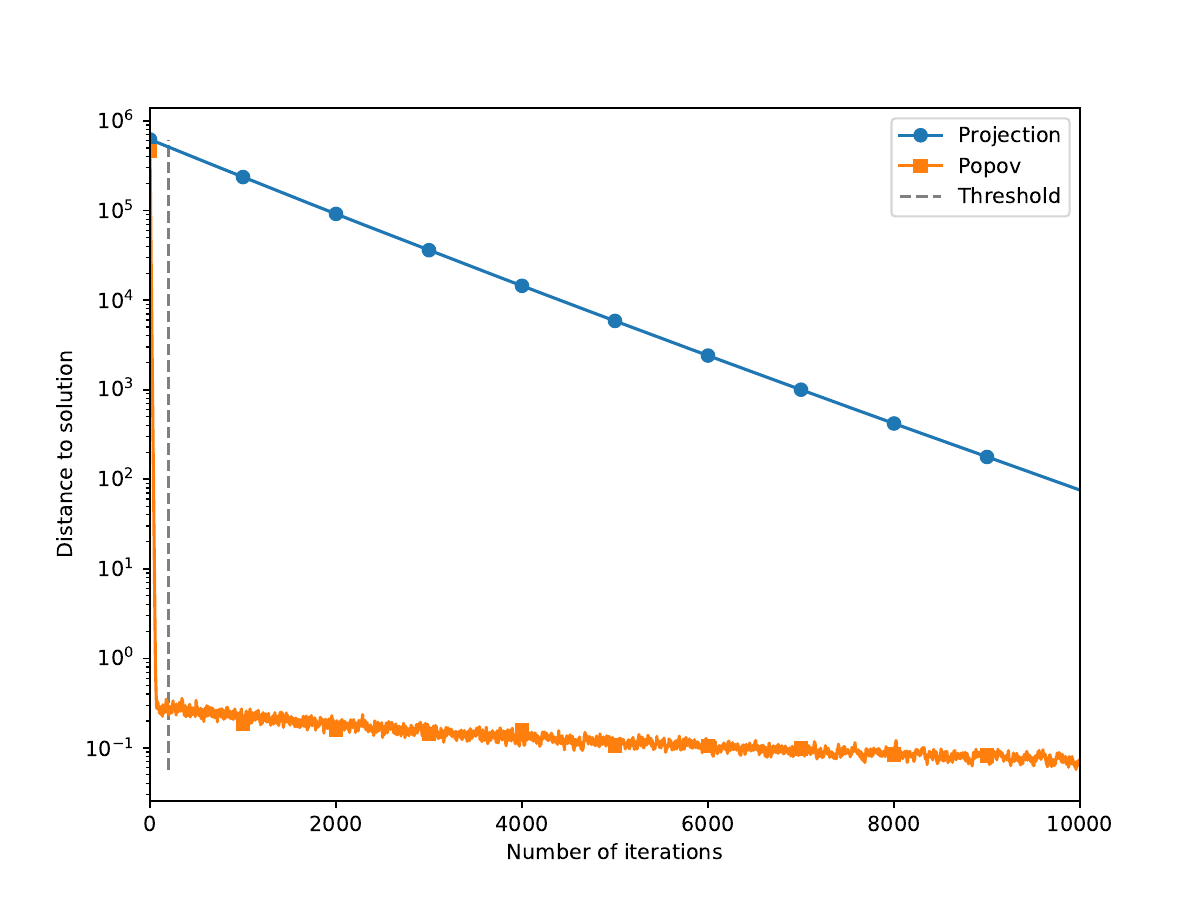}
}
\caption{Comparison of stochastic Popov method and stochastic projection method with stepsize rule given in~\eqref{eq-lemma3stitchstep} and $k_0=200$ for a finite-sum VI.}
\label{fig:stichstepsizes-a}
\end{figure*}

To formulate the preceding min-max problem as a finite-sum VI problem, we define $u = [u_1, u_2]'$ and the operator for every $i=1,\ldots,n$
\begin{equation}
    \label{eq-def-fi-a}
F_i(u) = \begin{bmatrix}
A_i & B_i\\
-B_i' & C_i 
\end{bmatrix} u +  \begin{bmatrix}
\mathbf{a}_i \\
\mathbf{c}_i
\end{bmatrix}\quad \hbox{for all } u\in\mathbb{R}^{\bd + \bs},\end{equation}
and we let 
\begin{equation}
    \label{eq-def-f-a}
    F(u) = \frac{1}{n} \sum_{i=1}^n F_i(u).\end{equation}
In this notation, the corresponding VI$(U,F)$ for the min-max problem consists of determining a point $u^*\in \mathbb{R}^{\bd + \bs}$ such that
\begin{equation}
\label{eq-numeric-2-a}
\langle F(u^*), u -  u^* \rangle  \ge 0 \ \text{for all } u \in U,\ \hbox{with } U= \mathbb{R}^{\bd + \bs}.
\end{equation}
We view the preceding VI$(U,F)$ problem
as a stochastic VI where
\[F(u)=\mathbb{E}[\Phi(u,\xi)],\]
where $\xi$ is a uniform random variable taking values in the set $\{1,2,\ldots,n\}$, with
\[\Phi(u,i)=f_i(u)\qquad\hbox{when $\xi=i$ \ for $i\in\{1,2,\ldots,n\}$}.\]

Since the constraint set is $U = \mathbb{R}^{\bd+\bs}$, the SVI$(U,F)$ in~\eqref{eq-numeric-2-a} reduces to the problem of determining a point $u^*\in \mathbb{R}^{\bd + \bs}$ such that $F(u^*) = 0$, i.e.\ 
$\mathbb{E}[\Phi(u^*,\xi)] =0.$

In our experiments, the number $n$ of random realizations of the uniform random variable $\xi$ is $n=20$.
For every $i=1,\ldots,n$, we generate positive definite symmetric matrices $A_i$ and $C_i$  with smallest eigenvalues $\mu_A>0$ and $\mu_C>0$,  and largest eigenvalues $L_A>0$ and $L_C>0$,  respectively. For all $i=1,\ldots,n$, to generate symmetric matrix $A_i$, firstly, we generate eigenvalues uniformly from on [$\mu_A, L_A$] such that $\mu_A, L_A$ are always generated. Then we generate a square random matrix $S_i \in \mathbb{R}^{\bd}$, do QR decomposition $S_i=Q_i R_i$, and get matrix $A_i$ as $A_i = Q_i \Lambda_i Q_i'$, where $\Lambda_i$ is a diagonal matrix with generated eigenvalues. We follow the same generation process for $C_i$ matrices.
For every $i=1,\ldots,n$, the entries of matrix $B_i$ and  vectors $\mathbf{a_i}, \mathbf{c_i}$ are sampled from a zero mean normal distribution with variances $\sigma_B^2 = 1 / (\bd + \bs)^2$, $\sigma_{bias}^2 = 1 / (\bd + \bs)$, respectively. We set threshold for both methods $k_0=200$, and parameters $\mu_{A}, \mu_{C} \in \{0.2, 0.02, 0.002 \}, L_{A} = L_{C} = 1$. The results are presented in Figure~\ref{fig:stichstepsizes-a}.

\subsection{Experiments on Example~\ref{example-p}}
Finally, we present experiments on Stochastic VI with the operator from Example~\ref{example-p} with different values $p$. Let $p>0$ and stochastic operator $\Phi: \mathbb{R}^2 \rightarrow \mathbb{R}^2$ be defined as
\begin{align*}
    \Phi(u, \xi) = c \begin{bmatrix} \text{\rm sign}(u_1) |u_1|^{p - 1} + u_2 \\ \text{\rm sign}(u_2) |u_2|^{p - 1} - u_1 \end{bmatrix} + \xi, \quad c= \left\{ \begin{array}{@{}cc} 2, \|u\| \leq 1, \\ 1,  \|u\| > 1 . \end{array} \right.
\end{align*} 
where $\xi$ is a random vector with entries being independent zero-mean Gaussian variables with variance $\sigma^2 = 1$.} 
We consider SVI$(U,F)$ where $U=\mathbb{R}^{2}$ and 
$F(u)=\mathbb{E}[\Phi(u,\xi)].$
This stochastic VI satisfies Assumption~\ref{asum-samples}, ~\ref{asum-Linear}, ~\ref{asum-sharp} when the samples $\{\xi_k\}$ are drawn independently according to the distribution of $\xi$. We study the performance of  Popov and projection methods across various values of $p$. Specifically, we set the parameters $p \in \{1.1, 1.3, 2.0\}$.  For each scenario, we choose {$K=10000$} and run twenty simulations, and the plots
show the performance for the average trajectories. To obtain \emph{a.s.} convergence as in Theorem~\ref{Theorem-Linear-sharp-ASC}, stepsizes for the Popov method is selected according to \eqref{eq-lemma3stitchstep} with $k_0=1$. Thus, stepsizes will satisfy Assumption~\ref{asum-steps} and by Theorem~\ref{Theorem-Linear-sharp-ASC} the method convergence to the solution \emph{almost surely}. Note that for projection method, there are no results on SVI with an operator satisfying linear growth and quasi-sharpness. Thus, stepsize for the projection method is chosen according to the rule given in~\eqref{eq-lemma3stitchstep}  with $k_0=1$,  and parameters $d =  \mu/C^2$, $a =  \mu$ as in  \cite{loizou2021stochastic}. The results are presented in Figure~\ref{fig:numeric-p-sharp}. Note that condition number $\kappa$ is small in all three scenarios. 
\begin{figure*}[ht]
\centering
\subfigure{
\includegraphics[width=.3\textwidth]{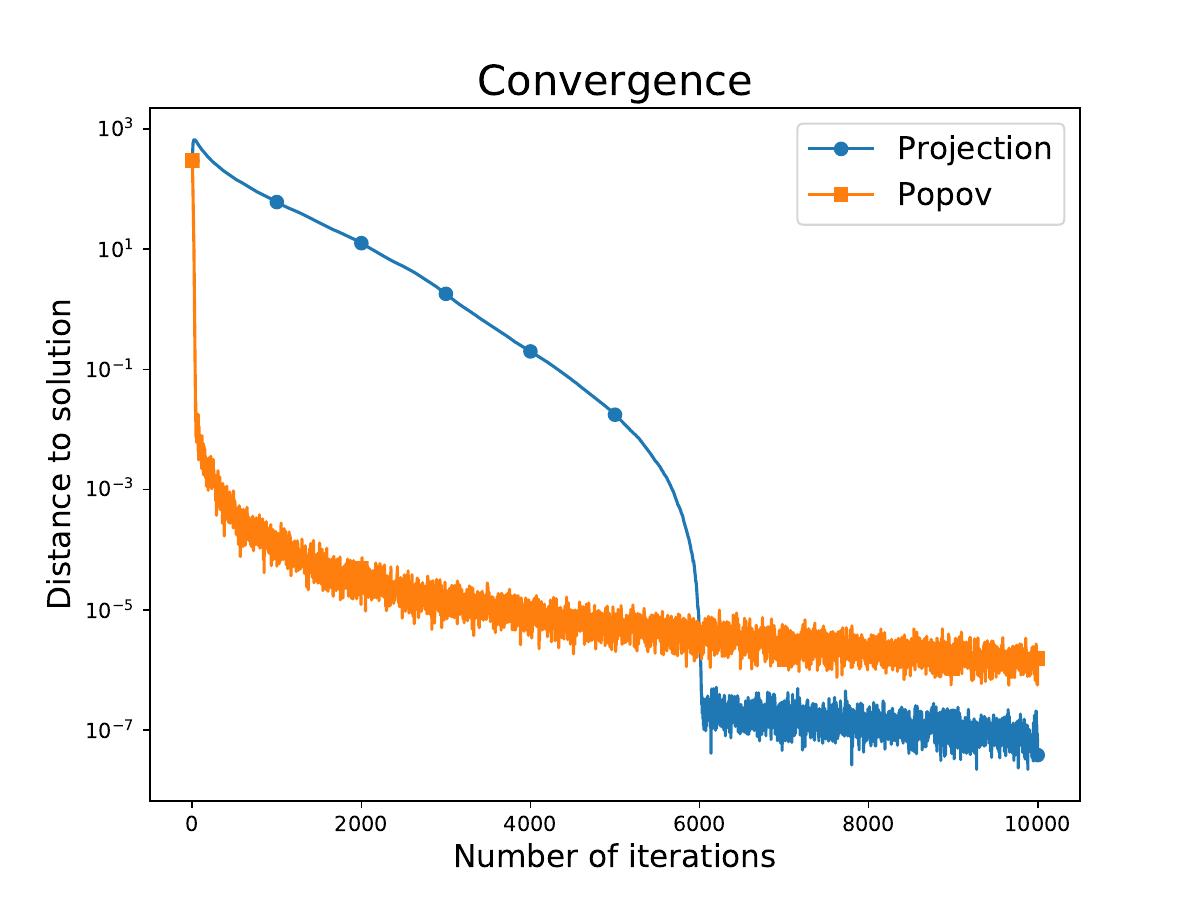}
}
\subfigure{
\includegraphics[width=.3\textwidth]{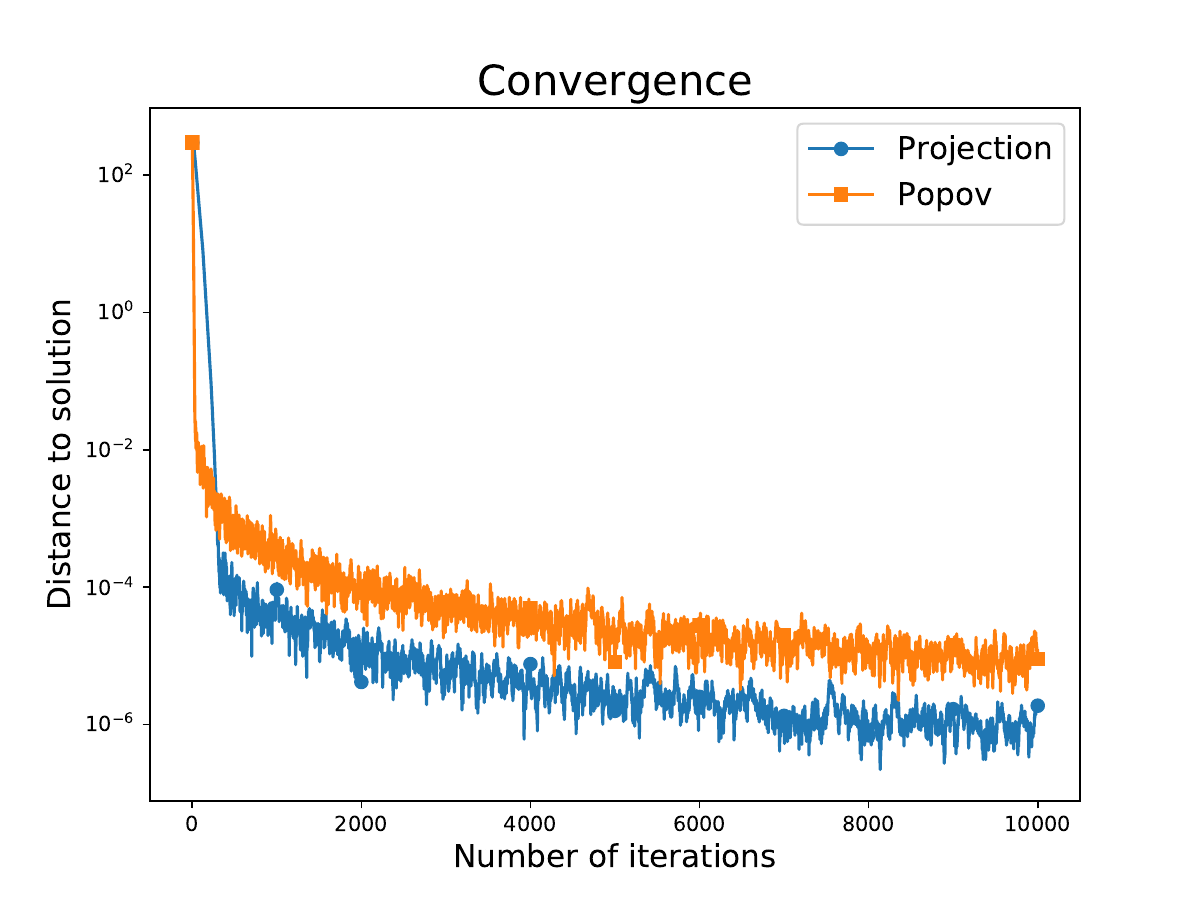}
}
\subfigure{
\includegraphics[width=.3\textwidth]{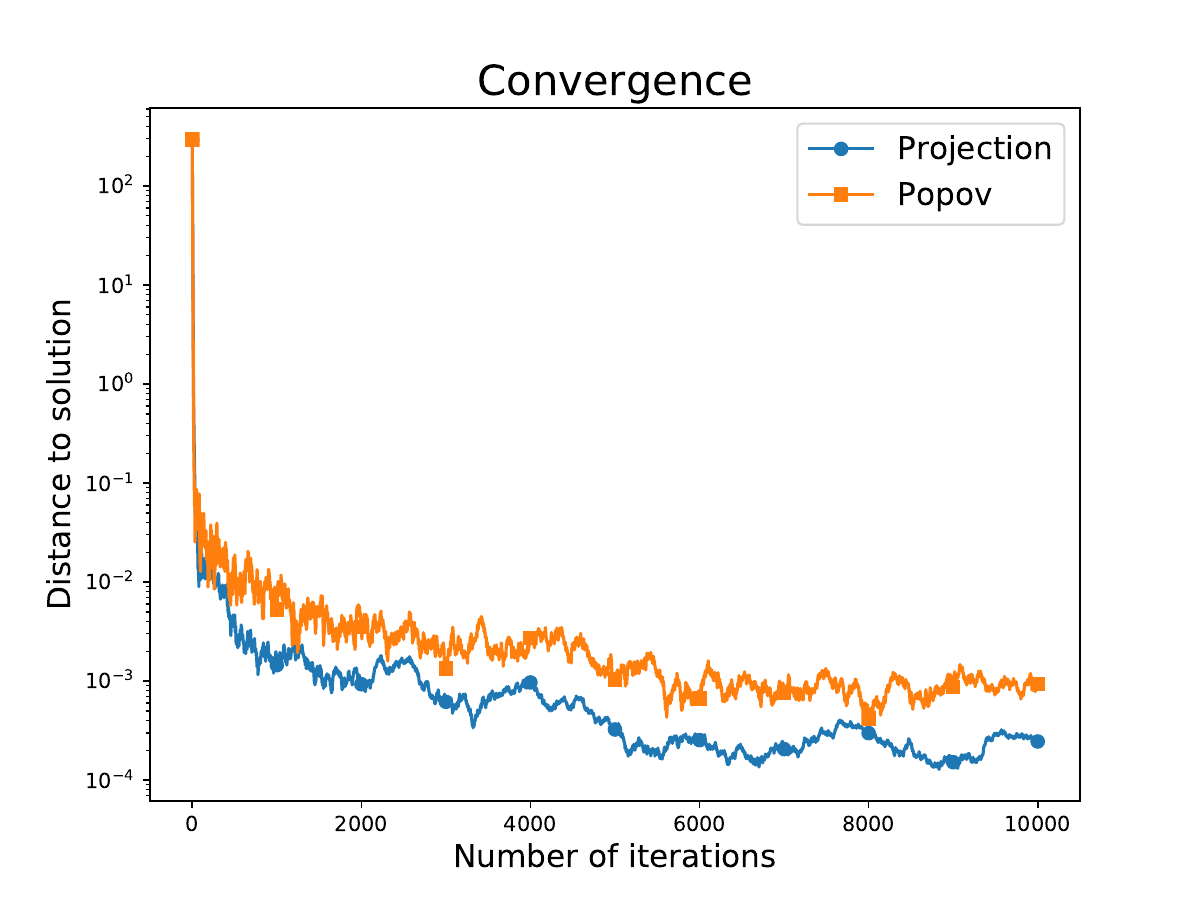}
} 
\hfill
\subfigure[$p=1.1$]{
\includegraphics[width=.3\textwidth]{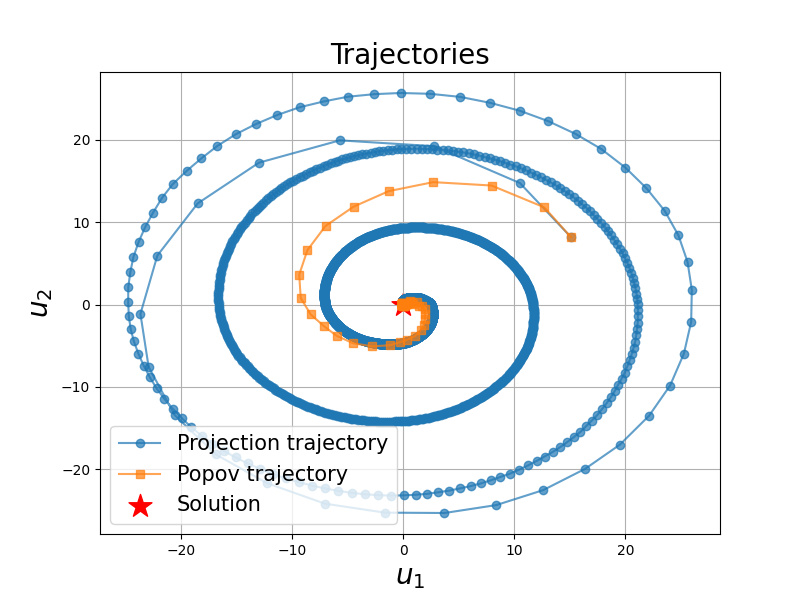}
}
\subfigure[$p=1.3$]{
\includegraphics[width=.3\textwidth]{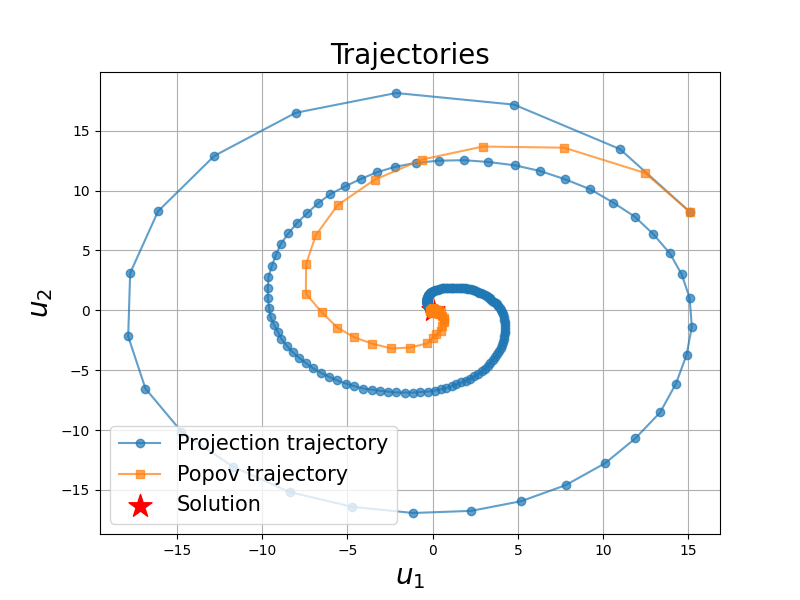}
}
\subfigure[$p=2.0$]{
\includegraphics[width=.3\textwidth]{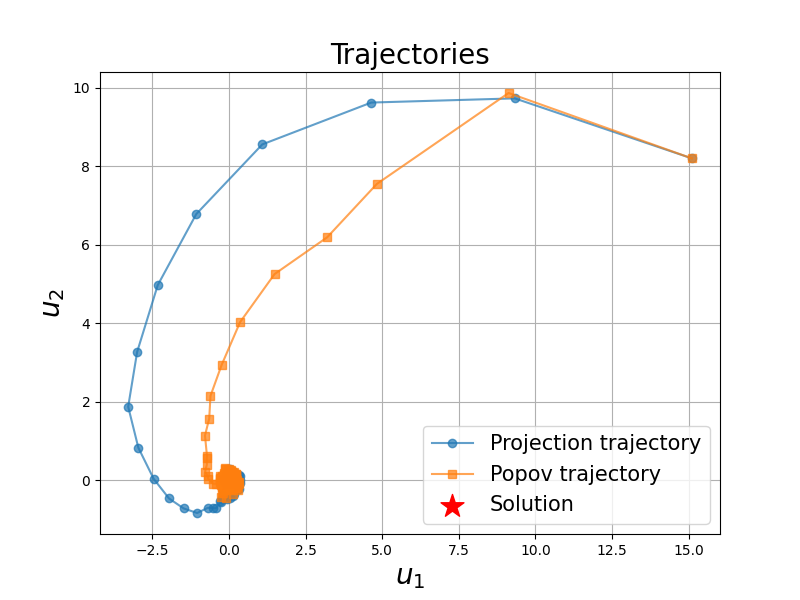}
}
\caption{Comparison of stochastic Popov method and stochastic projection method with stepsize rule given in~\eqref{eq-lemma3stitchstep} and $k_0=1$ with different values $p$.}
\label{fig:numeric-p-sharp}
\end{figure*}

\end{document}